\documentclass[11pt]{amsart}

\usepackage[latin1]{inputenc}
\usepackage{amsmath,amssymb,amsthm}
\usepackage{t1enc}

\usepackage[english]{}
\usepackage{graphicx}
\usepackage{epsfig}
\usepackage{graphicx}
\usepackage{multirow}
\usepackage{fancyhdr}
\usepackage{subfigure}
\usepackage{color}

\renewcommand{\leq}{\leqslant}

\def\N{\textrm{I\kern-0.21emN}}

\def\R{\textrm{I\kern-0.21emR}}
\def\Q{\textrm{l\kern-0.5emQ}}

\def\Z{\mathbb{Z}}

\newcommand{\ds}{\displaystyle}

\newcommand{\dive}{\mathop{\rm div}\nolimits}

\newcommand{\vsd}{\medbreak}
\newcommand{\dis}{\displaystyle}
\newcommand{\T}{\mathbb{T}}
\def\eqdefa{\buildrel\hbox{{\rm \footnotesize def}}\over =}

\newtheorem{theorem}{Theorem}[section]
\newtheorem{cor}[theorem]{Corollary}
\newtheorem{prop}[theorem]{Proposition}
\newtheorem{lemma}[theorem]{Lemma}

\theoremstyle{remark}
\theoremstyle{proof}
\newtheorem{remark}{Remark}[section]

\theoremstyle{definition}

\theoremstyle{notation}

\theoremstyle{definition}

\title{Wellposedness for density-dependent incompressible viscous fluids on the torus $\T^3$}
\author{Eugénie Poulon}
\date\today

\begin{document}

\synctex=1
\newcommand{\nfont}{\fontshape{n}\selectfont}
\noindent \address{{\nfont{Eug\'enie Poulon}} -  Laboratoire Jacques-Louis Lions - UMR 7598, Universit\'e Pierre et Marie Curie, Boite courrier 187, 4 place Jussieu, 75252 Paris Cedex 05, France}
\email{poulon@ann.jussieu.fr} 

\begin{abstract}
\noindent We investigate the local wellposedness of incompressible inhomogeneous Navier-Stokes equations on the Torus $\T^3$, with initial data in the critical Besov spaces. Under some smallness assumption on the velocity in the critical space $B^{\frac{1}{2}}_{2,1}(\T^3)$, the global-in-time existence of the solution is proved. The initial density is required to belong to~$B^{\frac{3}{2}}_{2,1}(\T^3)$ but not supposed to be small.
\end{abstract}
\maketitle
\section{Introduction and mains statements}
\noindent Incompressible flows are often modeled by the incompressible homogeneous Navier-Stokes system (\ref{NSH}), e.g the density of the fluid is supposed to be a constant
\begin{equation}
\label{NSH} \left \lbrace \begin {array}{ccc}  \partial_{t}v + v\cdot\nabla{v}-\Delta{v}&=&-\nabla{p}\\ \dive v&=&0\\
 v_{|t=0}&=&v_{0}.\\ \end{array}
\right.
\end{equation}
\noindent However, this model is sometimes far away from the physical situation. Concerning models of blood and rivers, even if the fluid is incompressible, its density can not be considered constant, owing to the complexity of the structure of the flow. As a result, a model which takes into account such constraints, has to be considered. That is the so-called Inhomogeneous Navier-Stokes system, given by 
\begin{equation} \label{NSIH1} \left\lbrace \begin {array}{ccc}  \partial_{t} \rho\, + \dive(\rho u) &=& 0\\ \partial_{t}(\rho u) + \dive(\rho u \otimes u) - 
\Delta{u} + \nabla{\Pi}&=&0\\
 \dive u&=&0\\
 (\rho,u)_{|t=0}&=&(\rho_{0},u_{0}).\\ \end{array}
\right.
\end{equation}
which is equivalent to the system below, by vertue of the transport equation
\begin{equation}
\label{NSIH} \left\lbrace \begin {array}{ccc}  \partial_{t} \rho\, + u\cdotp\nabla{\rho} &=& 0\\ \rho(\partial_{t} u + u\cdot\nabla{u})  - \Delta{u} + \nabla{\Pi}&=&0\\
 \dive u&=&0\\
 (\rho,u)_{|t=0}&=&(\rho_{0},u_{0}),\\ \end{array}
\right.
\end{equation}
where $\rho=\rho(t,x) \in \R^{+}$ stands for the density and $u=u(t,x) \in \T^3$ for the velocity field. The term $\nabla \Pi$ (namely the gradient of the pressure) may be seen as the Lagrangian multiplier associated with the constraint $\dive u=0$. The initial data ($\rho_{0},u_{0}$) are prescribed. Notice, we choose the viscosity of the fluid equal to~$1$, in a sake of simplicity.\\ \\
\noindent Let us recall some well-known results about the two above systems (homogeneous versus inhomogeneous). In the homogeneous case, the celebrated theorem of J. Leray \cite{JLbis} proves the global existence of weak solutions with finite energy in any space dimension. The uniqueness is garanteed in dimension $2$, whereas in dimension~$3$, this is still an open question. In deal with this issue, H. Fujita and T. Kato \cite{FKbis} built some global strong solutions in the context of scaling invariance spaces, namely spaces which have the same scaling as the system (\ref{NSH}). Such spaces are said to be critical, in the sense that their norm is invariant for any~$\lambda >0$ under the transformation 
$$ v_{0}(x) \mapsto \lambda\,v_{0}(\lambda x) \quad \hbox{and} \quad v(t,x) \mapsto \lambda\,v(\lambda^2 t,\lambda x).$$ The point is that such solutions are unique in this framework. In the inhomogeneous case, Leray's approach is still relevant for the system (\ref{NSIH1}). Indeed, if the initial density~$\rho_{0}$ is non negative and belongs to~$L^\infty$ and if $\sqrt{\rho_{0}}\, u_{0}$ belongs to~$L^2$, then there exists some global weak solutions $(\rho,u)$ with finite energy. However, the question of uniqueness has not been solved, even in dimension $2$. We refer the reader to the paper of A. Kazhikhov \cite{Kaz}, J. Simon \cite{S} for the existence of global weak solutions. The unique resolvability of (\ref{NSIH1}) is first established by the works of O. Ladyzenskaja and V. Solonnikov \cite{LS} in the case of a bounded domain~$\Omega$ with homogeneous Dirichlet condition for the velocity $u$. As one has already mentionned previously, the approach initiated by H. Fujita and T. Kato is particulary efficient in the scaling invariance framework to face the uniqueness problem. A natural question is to wonder if such an approach is relevant for incompressible inhomogeneous fluids. If one believes so, scaling considerations should help us to find an adaptated functional framework. Firstly, one can check that (\ref{NSIH}) is invariant under the scaling transformation (for any $\lambda>0$) 
$$ (\rho_{0},u_{0})(x) \mapsto (\rho_{0},\lambda\,u_{0})(\lambda x) \quad \hbox{and} \quad (\rho,u,\Pi)(t,x) \mapsto (\rho,\,\lambda\,u,\, \lambda^2\,\Pi)(\lambda^2 t,\lambda x).$$
\noindent That is an easy exercice to check that $\ds{\dot{B}^{\frac{3}{2}}_{2,1}(\R^3) \times \dot{B}^{\frac{1}{2}}_{2,1}(\R^3)}$ is scaling invariant under this transformation, in dimension $3$, e.g
$$ \| \rho_{0}(\lambda x) \|_{\dot{B}^{\frac{3}{2}}_{2,1}(\R^3)} =  \|  \rho_{0}\|_{\dot{B}^{\frac{3}{2}}_{2,1}(\R^3)} \quad \hbox{and} \quad \|  \lambda u_{0}(\lambda x) \|_{\dot{B}^{\frac{1}{2}}_{2,1}(\R^3)} =\|   u_{0} \|_{\dot{B}^{\frac{1}{2}}_{2,1}(\R^3)}.$$ 
Secondly, as the system (\ref{NSIH}) degenerates if $\rho$ vanishes or becomes unbounded, we further assume that the density is away from zero ($\rho^{\pm1}_{0} \in L^{\infty}$). Denoting 
$$ \frac{1}{\rho_{0}} \eqdefa 1 +  a_{0} \quad \hbox{and} \quad \frac{1}{\rho} \eqdefa 1 +a,$$
the incompressible inhomogeneous Navier-Stokes system (\ref{NSIH}) can be rewritten as     
\begin{equation}
\label{NSIH avec a} \left \lbrace \begin {array}{ccc}  \partial_{t} a\, + u\cdotp\nabla{a} &=& 0\\ \partial_{t} u + u\cdot\nabla{u} +(1+a)\, ( \nabla{\Pi} - \Delta{u} )&=&0\\
 \dive u&=&0\\
 (a,u)_{|t=0}&=&(a_{0},u_{0}),\\ \end{array}
\right.
\end{equation}
\medbreak\noindent The question of unique solvability of the above system (\ref{NSIH avec a}) has been adressed by many authors. Let us highlight the work of R. Danchin \cite{Dan1}, who studied the unique solvability of (\ref{NSIH avec a}) with constant viscosity coefficient and in scaling invariant (e.g critical) Besov spaces in the whole space~$\R^N$. This generalized the celebrated results by H. Fujita and T. Kato, devoted to the classical homogeneous Navier-Stokes system (\ref{NSH}). Indeed, R. Danchin proved in \cite{Dan1} (under the assumption the density is close to a constant) a local well-posedness for large initial velocity and a global well-posedness for initial velocity small with respect to the viscosity. More precisely, he proved that if the initial data~$(a_{0},u_{0})$ belongs to $\dot{B}^{\frac{N}{2}}_{2,\infty}(\R^N)\,\cap\, L^{\infty}(\R^N) \times \dot{B}^{\frac{N}{2}-1}_{2,1}(\R^N)$, with~$a_{0}$ small enough in~$\dot{B}^{\frac{N}{2}}_{2,\infty}(\R^N)\cap L^{\infty}(\R^N)$, then the system (\ref{NSIH avec a}) has a unique local-in-time solution. In addition, assuming the velocity~$u_{0}$ is also small enough in the space~$\dot{B}^{\frac{N}{2}-1}_{2,1}(\R^N)$, the solution is global.   
\medbreak
\noindent Our main motivation in this paper is to investigate the local and global wellposedness of the incompressible inhomogeneous Navier-Stokes system, in the case of critical Besov spaces and on the torus~$\T^3$. The aim is to get rid of the smallness condition on the density, and just keeping the smallness one on the initial velocity. We point out that such a result has been already proved in the whole space~$\R^3$. We refer the reader to the paper \cite{AGZ2} of H. Abidi, G. Gui and P. Zhang. The main difference between their work and ours is that, on the torus, we have to be careful, owing to the average of the velocity~$u$, which is not preserved, contrary to the case of classical Navier-Stokes system (\ref{NSH}). As a consequence, a lot of "classical results" such as Gagliardo-Niremberg inequalities and Sobolev embeddings, have to take into account the average of the velocity $u$. We will collect them in section $2$. Let us give some remarks about this.\\
\noindent \textbf{Notation}
\noindent In the sequel, we shall denote by $$ \bar{m}\eqdefa \int_{\T^3} m(x)\, dx,  \quad \hbox{where} \quad | \T^3 |=1.$$
\begin{remark}
\label{argument 1}
It is clear that $\bar{\rho} = \bar{\rho_{0}}$. Indeed, an integration on the mass conservation equation combining with the fact $\dis{\int_{\T^3} u\cdotp\nabla{\rho} =0}$ gives
$$ \int_{\T^3} \rho(t,x)\, dx \, = \, \int_{\T^3} \rho_{0}(x)\, dx.$$
\end{remark}
\noindent Notice that by vertue of the divergence free condition on the velocity $u$, the average of any function of~$\rho$ is preserved. In particular, the average of~$a$ is conserved. \\
\begin{remark}
\label{argument 2}
An integration on the momentum equation of the system (\ref{NSIH1}) (the terms $\ds{\int_{\T^3}\dive(\rho u \otimes u)}$, $\dis{\int_{\T^3} 
\Delta{u}}$ and $\dis{\int_{\T^3} \nabla{\Pi}}$ are nul) implies
$$ \int_{\T^3} (\rho\,u) (t,x)\, dx = \int_{\T^3} \rho_{0}\,u_{0} (x)\, dx.$$ 
\end{remark}

\begin{remark}
\label{argument 3}
Notice that~$ \rho - \bar{\rho}$ is also solution of the transport equation. Thus, if we take the~$L^2$ inner product of this mass conservation equation with~$\rho - \bar{\rho}$ itself, we get the energy conservation of the quantity~$\| \rho - \bar{\rho}\|_{L^2}$, because of divergence-free condition of $u$. Therefore we have : $$\| \rho - \bar{\rho}\|_{L^2} = \| \rho_{0} - \bar{\rho_{0}}\|_{L^2}.$$ 
\end{remark}

\noindent In this paper, our main Theorem can be stated as follows 
\begin{theorem}[Main theorem]\,\,
\label{main theorem}
\sl{
Let $a_{0} \in B^{\frac{3}{2}}_{2,1}$, $u_{0} \in B^{\frac{1}{2}}_{2,1}$, such that 
\begin{equation}
\dive u_{0}=0 \quad \hbox{;} \quad 1 + a_{0} \geqslant b \quad \hbox{for some positive constant}  \,\, b \quad \hbox{and}  \quad \int_{\T^3} \frac{1}{1 + a_{0}(x)} \, u_{0}(x) \, dx =0.
\end{equation} Then there exists a positive time $T_{*}$ such that the system (\ref{NSIH avec a}) has a unique local-in-time solution : for any~$T <T_{*}$, $$(a,u,\Pi)\,\, \in \,\, \mathcal{C}([0,T],B^{\frac{3}{2}}_{2,1})\, \times\, \mathcal{C}([0,T],B^{\frac{1}{2}}_{2,1})\,\cap\, L^{1}([0,T],B^{\frac{5}{2}}_{2,1}) \times\, L^{1}([0,T],B^{\frac{1}{2}}_{2,1})\cdotp$$ \noindent In addition, there exists a constant $c$ (depending on $\ds{\| a_{0}\|_{B^{\frac{3}{2}}_{2,1}}}$) such that $$ \hbox{if} \quad \| u_{0}\|_{B^{\frac{1}{2}}_{2,1}} \leqslant c, \quad  \hbox{then} \quad  T_{*} = +\infty.$$
}\end{theorem}

\medbreak
\noindent Our main Theorem \ref{main theorem} relies on two Theorems, given below. Indeed, we will face the question of local wellposedness and global wellposedness in a different way. The first one deals with the local wellposed issue: until a small time, we may control the velocity $u$ in some functional Besov spaces, by the initial data $u_{0}$. It can stated as follows 
\begin{theorem}[Local-wellposedness theorem]\,\,
\label{theorem LWP}
\sl{
Let $a_{0} \in B^{\frac{3}{2}}_{2,1}$, $u_{0} \in B^{\frac{1}{2}}_{2,1}$, such that 
\begin{equation}
\dive u_{0}=0 \quad \hbox{;} \quad 1 + a_{0} \geqslant b \quad \hbox{for some positive constant}  \,\, b.
\end{equation} Then there exists a positive time $T_{*}$ such that the system (\ref{NSIH avec a}) has a unique local-in-time solution : for any~$T <T_{*}$, $$(a,u,\Pi)\,\, \in \,\, \mathcal{C}([0,T],B^{\frac{3}{2}}_{2,1})\, \times\, \mathcal{C}([0,T],B^{\frac{1}{2}}_{2,1})\,\cap\, L^{1}([0,T],B^{\frac{5}{2}}_{2,1}) \times\, L^{1}([0,T],B^{\frac{1}{2}}_{2,1})\cdotp$$ \noindent In addition, there exists a small constant $c$ depending on $\ds{\| a_{0}\|_{B^{\frac{3}{2}}_{2,1}}}$ such that if $$ \| u_{0}\|_{B^{\frac{1}{2}}_{2,1}} \leqslant c, $$
therefore,  $T_{*} \geqslant 1$ and one has for any $T < T_{*}$,
\begin{equation}
\label{density}
\hbox{Density estimate:}\quad \|a \|_{L^{\infty}_{T}(B^{\frac{3}{2}}_{2,1})} \leqslant \|a_{0}\|_{B^{\frac{3}{2}}_{2,1}}\,\exp\Bigl(C\, \|u \|_{L^{1}_{T}(B^{\frac{5}{2}}_{2,1})}\Bigr).  
\end{equation}
\begin{equation}
\label{velocity}
\hbox{Velocity estimate:}\quad \|u \|_{L^{\infty}_{T}(B^{\frac{1}{2}}_{2,1})} + \|u \|_{L^{1}_{T}(B^{\frac{5}{2}}_{2,1})} + \|\nabla \Pi \|_{L^{1}_{T}(B^{\frac{1}{2}}_{2,1})} \leqslant C\, \|u_{0}\|_{B^{\frac{1}{2}}_{2,1}}.
\end{equation}   
}\end{theorem}

\medbreak

\begin{remark}
The difficulty, as mentionned previously, is that the density $a$ is not supposed to be small. To overcome this issue, we split the density $1+a$ into 
$$ 1+ a = (1 + S_{m}a) + (a -S_{m}a), \quad  \hbox{where} \quad S_{m}a \eqdefa \sum_{j \leqslant m-1} \Delta_{j}a.$$
The first part is then regular enough, the second part can be made small enough, for some large enough integer $m$: we fix $m$ in the sequel such that $\| a - S_{m}a \|_{B^{\frac{3}{2}}_{2,1}} \leqslant c$. 
\end{remark}

\noindent The local wellposedness Theorem \ref{theorem LWP} is an immediate consequence of Lemma below, which will be useful in the sequel. 

\begin{lemma}\sl{
\label{lemma general avec linfini}
 Let $T>0$ be a fixed finite time. For any $t \in [0,T]$, the velocity estimate is given by
\begin{equation}
\label{velocity bis}
\|u \|_{L^{\infty}_{t}(B^{\frac{1}{2}}_{2,1})} + \|u \|_{L^{1}_{t}(B^{\frac{5}{2}}_{2,1})} + \|\nabla \Pi \|_{L^{1}_{t}(B^{\frac{1}{2}}_{2,1})} \leqslant C\, \|u_{0}\|_{B^{\frac{1}{2}}_{2,1}} + \int_{0}^{t} \, \bigl( \|\nabla u(t') \|_{L^{\infty}} \, + \, W(t')  \bigr) \|u(t') \|_{B^{\frac{1}{2}}_{2,1}} dt', 
\end{equation}  
\noindent where $$ W(t') \eqdefa 2^{2m}\, \| a\|^{2}_{L^{\infty}_{t'}(L^{\infty})} \, +\, 2^{8m}\, \|a\|^{4}_{L^{\infty}_{t'}(L^{2})}\, \bigl( 1 +\, \|u\|^{4}_{L^{\infty}_{t'}(B^{\frac{1}{2}}_{2,1})}\, + \| a\|^{4}_{L^{\infty}_{t'}(L^{\infty})}   \bigr).$$  
}\end{lemma}
\vskip 0.2cm
\noindent Two above results will provide us the local and uniqueness existence of a solution $(a,u)$. Concerning the global aspect to this solution, we shall use an energy method, which can be achieved by vertue of Theorem \ref{theorem1GWP} below.
\begin{theorem}[Global wellposedness Theorem]\;\,
\label{theorem1GWP}\sl{
Given the initial data $(\rho_{0},u_{0})$ and two positive constants $m$ and $M$ such that 
\begin{equation}
u_{0} \in  H^{2}(\T^3), \quad  0 < m \leqslant \rho_{0}(x) \leqslant M, \quad \hbox{and} \quad \int_{\T^3} \rho_{0}\,u_{0}=0.
\end{equation} 
There exists a constant $\varepsilon_{0}>0$ (depending on $m$ and $M$)  such that if $u_{0}$ satisfies the smallness condition $\ds{ \|  u_{0}\|_{H^2} \leqslant \varepsilon_{0}}$  
then, the system (\ref{NSIH}) has a (unique) global solution $(\rho,u)$ which satisfies for any $(t,x) \in [0,+\infty[ \times \T^3$ 
\begin{equation}
\begin{split}
&0 < m \leqslant \rho(t,x) \leqslant M,\\
& B_{0}(t) \leqslant \| \sqrt{\rho_{0}} u_{0}\|^{2}_{L^2}, \\
& B_{1}(t) \leqslant C\,  \| \nabla u_{0} \|^{2}_{L^2},\\
&B_{2}(t) \leqslant\, C\, \Bigl(1\, +\,  \|  u_{0} \|^4_{H^2}\Bigr)\,\, \|  u_{0}\|^2_{H^2}\,\,
\exp{\Bigl(  \|  u_{0}\|^2_{L^2} + \|  \nabla u_{0}\|^2_{L^2} \Bigr)}\\  
\end{split}
\end{equation}
where $ B_{0}(t)$, $ B_{1}(t)$ and $ B_{2}(t)$ are defined by 
\begin{equation}
\label{L2 energy estimate}
\begin{split}\
 B_{0}(T) &\eqdefa\, \sup_{t \in [0,T]} \| \sqrt{\rho} u(t)\|^{2}_{L^2} \, + \, \int_{0}^{T} \int_{\T^3} | \nabla u(t,x)|^2 \,dx\,dt. \\
\end{split}
\end{equation}
\begin{equation}
\label{H1 energy estimate}
\begin{split}\
& B_{1}(T) \eqdefa\,  \,\sup_{t \in [0,T]} \| \nabla{u}(t)\|^{2}_{L^2} + \int_{0}^{T} \Bigl(\| \sqrt{\rho}\,\partial_{t} u(t)\|^{2}_{L^2} + \| \nabla{\Pi}(t)\|^{2}_{L^2} \Bigr) dt\, + \frac{1}{8}\int_{0}^{T} \| \nabla^2{u}(t)\|^{2}_{L^2}\, dt,  
\end{split}
\end{equation} 
\begin{equation}
\label{H2 energy estimate}
\begin{split}
 B_{2}(T) &\eqdefa\,  \, \sup_{t \in [0,T]} \,  \Bigl(\frac{1}{2} \|\nabla^2{u}(t)\|^2_{L^2} + \| \nabla{\Pi}(t)\|^2_{L^2} + \frac{m}{3}\, \| \partial_{t} u(t) \|^2_{L^2}\Bigr)\\ &\qquad \qquad +\  \frac{1}{4} \int_{0}^{T} \, \|\nabla{\partial_{t} u(t)}\|^2_{L^2} \, dt + \frac{1}{2} \int_{0}^{T} \, \|\nabla^2{u}(t)\|^2_{L^6}  dt\,+ \,\int_{0}^{T} \, \| \nabla^2 \Pi(t) \|^2_{L^6}\, dt.
\end{split}
\end{equation} 
}\end{theorem} 
\begin{remark}
\noindent We shall prove the existence and global part by an energy method.  We underline the very weak assumption (bounded from above and below) on the density we need. We refer the reader to \cite{PZZ} for the uniqueness proof. 
\end{remark}
\medskip
\noindent \textit{Guideline of the proof and organisation of the paper.}\\
Firstly, we prove the local existence and uniqueness of a solution, under hypothesis of Theorem \ref{theorem LWP}. Then, we underlinde that, provided $\|u_{0}\|_{B^{\frac{1}{2}}_{2,1}}$ is small enough, the lifespan  $T^{*}(u_{0})$ of the local solution associated with this data should be greater than $1$. This is due to scaling argument. In addition, velocity estimate (\ref{velocity}) implies
\begin{equation}
\label{petitesse sur u(t1)}
\exists t_{1} \in [0,1[\, \quad \hbox{such that} \quad u(t_{1}) \in H^2 \quad \hbox{and} \quad \|u(t_{1})\|_{H^2} \leqslant C\, \|u_{0}\|_{B^{\frac{1}{2}}_{2,1}} . 
\end{equation} 
This stems from an interpolation argument, provided~$T^{*}(u_{0}) >1$. Indeed, assume we have proved there exists an unique solution $u$ such that 
$$ u \in  L^{\infty}([0,T],B^{\frac{1}{2}}_{2,1}) \cap L^{1}([0,T], B^{\frac{5}{2}}_{2,1}),$$ and thus, $u$ belongs to $L^{\frac{4}{3}}([0,T],H^{2})$, which provide the existence of the small time~$t_{1}$, such that (\ref{petitesse sur u(t1)}) is satisfied. \\
\noindent From this point, the strategy to deal with the global property of our system takes another direction than the strategy setting up in \cite{AGZ2}. Indeed, we shall prove that, considering~$u(t_{1})$ as an initial data in~$H^2$, which is small enough (since~$ \|u_{0}\|_{B^{\frac{1}{2}}_{2,1}}$ is supposed to be so) and thanks to Theorem \ref{theorem1GWP} below, there exists a global solution (the uniqueness is non necessary for what we need in the sequel).\\Then, it remains to be seen that such a solution has the relevant regularity, namely the regularity demanding by Theorem \ref{theorem LWP}. In others words, it is crucial to prove the propagation of the regularity of the density function $a$, from which we infer the regularity of the velocity, thanks to Lemma \ref{lemma general avec linfini}. To sum up, we will prove the existence of a global solution with the relevant regularity : this proves the uniqueness of such a solution.

\medskip 
\noindent The paper is structured as follows. In Section $2$, we collect some basic facts on Littlewood Paley theory, Besov spaces and we will give the classical inequalities (well-known in the whole space $\R^3$), in the case of the torus $\T^3$. In addition, we will stress on the important role of the average $u$.\\ Section $3$ is devoted to the proof of the main Theorem \ref{main theorem}. Section $4$ deals with the local wellposedness issue of the main theorem : we will prove Theorem \ref{theorem LWP}. Section $5$ provides the global wellposedness aspect of the main theorem, which will stem from the proof of Theorem \ref{theorem1GWP}. Let us mention we will only give in both two cases the a priori estimates. It means we skip the standard procedure of Friedrich's regularization. The point is that we deal with uniform estimates, in which we use a standard compactness argument.   

\medbreak 
\section{Tool box concerning estimates on the Torus $\T^3$}
\begin{prop}(Poincar\'e-Wirtinger inequality)\\
Let $u$ be in $H^1(\T^3)$ and mean free. Then we have :
$$ \| u\|_{L^2(\T^3)} \leqslant \| \nabla{u} \|_{L^2(\T^3)}. $$  
In particular, the $\dot{H}^1(\T^3)$ and $H^1(\T^3)$-norms are equivalent, when $\bar{u}$ is mean free. 
\end{prop}
\vskip 0.3cm
\noindent An obvious consequence of the Poincar\'e-Wirtinger inequality is the corollary below.
\begin{cor}
Let $u$ be in $H^1(\T^3)$. Then we have :
$$ \| u - \bar{u} \|_{L^2(\T^3)} \leqslant \| \nabla{u} \|_{L^2(\T^3)}. $$ 
\end{cor}

\vskip 0.3cm
\begin{prop}(Gagliardo-Niremberg inequality)
\label{GN}
$$ \hbox {In the whole space}\quad  \R^3 : \quad \| u \|_{L^p} \leqslant \|u\|^{\frac{3}{p}-\frac{1}{2}}_{L^2}\, \| \nabla{u}\|^{\frac{3}{2}-\frac{3}{p}}_{L^2},\quad \hbox{with} \quad 2\leqslant p \leq 6.$$ 
$$ \hbox {On the torus}\quad \T^3 : \quad \| u - \bar{u} \|_{L^p} \leqslant \|u\|^{\frac{3}{p}-\frac{1}{2}}_{L^2}\, \| \nabla{u}\|^{\frac{3}{2}-\frac{3}{p}}_{L^2},\quad \hbox{with} \quad  2\leqslant p \leq 6.$$
\end{prop}
\noindent In particular, for $p=6$, we find the Sobolev embeddings on the torus : 
 $$ \| u  - \bar{u} \|_{L^6(\T^3)} \leqslant C\,\| \nabla u\|_{L^2(\T^3)} \quad \hbox{instead of} \quad \| u\|_{L^6(\R^3)} \leqslant C\,\| \nabla u\|_{L^2(\R^3)}.$$

\vskip 0.2cm
\noindent The following Lemma is fundamental in this paper. It highlights the crucial role playing by the average of the velocity. Because the framework of our work is the torus, we will need several times in the next, to have an estimate on the average. Actually, it provides a general method to compute the average of a quantity we are intesresting in. We will call it \textit{the average method} in the sequel.
\begin{lemma} 
\label{average}
Assuming that $|\T^3| = 1$ and $\dis{\int_{\T^3} \rho_{0}\,u_{0}=0}$, we have : 
$$|\bar{u}(t)| \leqslant \frac{\| \rho_{0} - \bar{\rho_{0}}\|_{L^2}}{|\bar{\rho_{0}}|}\, \| \nabla{u}(t)\|_{L^2}.$$
\end{lemma}
\vsd
\begin{proof} Let us consider the integral below and developp it
$$\int_{\T^3} (\rho - \bar{\rho})(t,x)( u - \bar{u})(t,x)\, dx = \int_{\T^3} \rho(t,x)\,u(t,x) -2\bar{\rho}(t)\,\bar{u}(t)\, + \bar{\rho}(t)\,\bar{u}(t).$$
Thanks to (\ref{argument 1})  and (\ref{argument 2}), we have
\begin{equation}
\begin{split}
\bar{u}(t) &= -\frac{1}{\bar{\rho}(t)} \int_{\T^3} (\rho - \bar{\rho})(t,x)( u - \bar{u})(t,x)\, dx\\
=& -\frac{1}{\bar{\rho_{0}}} \int_{\T^3} (\rho - \bar{\rho})(t,x)( u - \bar{u})(t)\\
|\bar{u}(t)|&\leqslant \frac{1}{|\bar{\rho_{0}}|} \, \| (\rho - \bar{\rho})(t)\|_{L^2}\,\| (u - \bar{u})(t)\|_{L^2}.\\
\end{split}\end{equation}
Applying (\ref{argument 3}), we have
\begin{equation}
\begin{split}
|\bar{u}(t)|&\leqslant \frac{1}{|\bar{\rho_{0}}|} \, \| \rho_{0} - \bar{\rho_{0}}\|_{L^2}\,\| (u - \bar{u})(t)\|_{L^2}.\\
\end{split}\end{equation}
Thanks to Poincar\'e-Wirtinger, we get :
\begin{equation}
|\bar{u}(t)| \leqslant \frac{\| \rho_{0} - \bar{\rho_{0}}\|_{L^2}\,}{|\bar{\rho_{0}}|} \,\| \nabla{u}(t)\|_{L^2}.\\
\end{equation}
\end{proof}

\vskip 0.3cm 
\begin{prop}
\label{application1 lemme average}
Assuming that $|\T^3| = 1$ and $\dis{\int_{\T^3} \rho_{0}\,u_{0}=0}$, \,\,\, therefore \, \,\,
$\ds{\| u(t)\|_{L^6} \leqslant  C(\rho_{0})\, \|\nabla u(t)\|_{L^2}}$.
\end{prop}
\begin{proof}
\begin{equation}
\begin{split}
\| u(t)  \|^2_{L^6} &\leqslant  \| (u - \bar{u})(t)  \|^2_{L^6} + | \bar{u}(t) |^2\\
&\leqslant C\, \|\nabla u(t)\|^2_{L^2} + \frac{\| \rho_{0} - \bar{\rho_{0}}\|^2_{L^2}}{\bar{\rho_{0}}^2}\, \| \nabla{u}(t)\|^2_{L^2}\\
&\leqslant C(\rho_{0})\, \|\nabla u(t)\|^2_{L^2}.
\end{split}\end{equation}
\end{proof}
\vskip 0.3cm
\begin{prop}
\label{application 2 lemme average}
If $|\T^3| = 1$ and $\dis{\int_{\T^3} \rho_{0}\,u_{0}=0}$, \,\, \, then \, \, \,
$\ds{\|  u(t) \|_{L^3} \leqslant C(\rho_{0})\, \|  \nabla u(t) \|_{L^2}}$.
\end{prop}

\begin{proof}
Arguments are similar as before. We introduce the average of $u$ and we apply succesively Gagliardo-Niremberg and Poincar\'e-Wirtinger inequalities  
\begin{equation}
\begin{split}
\|  u(t) \|_{L^3} &\leqslant \|  (u - \bar{u})(t) \|_{L^3} + | \bar{u}(t) |\\
& \leqslant \|  (u - \bar{u})(t) \|^{\frac{1}{2}}_{L^2}\, \| \nabla( u - \bar{u})(t) \|^{\frac{1}{2}}_{L^2} + | \bar{u}(t) |\\
&\leqslant \|  \nabla u(t)  \|^{\frac{1}{2}}_{L^2}\, \|  \nabla u(t)  \|^{\frac{1}{2}}_{L^2}\, + | \bar{u}(t) |\\
&\leqslant \|  \nabla u(t)  \|_{L^2}\, + | \bar{u}(t) |.\\ 
\end{split}
\end{equation}
Concerning the term $| \bar{u}(t) |$, same computations as in Lemma \ref{average} yield 
\begin{equation}
\begin{split}
\|  u(t) \|_{L^3} &\leqslant \|  \nabla u(t)  \|_{L^2}\, + \frac{1}{|\bar{\rho_{0}}|}\,\, \|\rho_{0} - \bar{\rho_{0}} \|_{L^2}\,\, \|  (u - \bar{u})(t) \|_{L^2}  \\
& \leqslant C(\rho_{0})\, \|  \nabla u(t)  \|_{L^2}.
\end{split}
\end{equation}
\end{proof}

\section{Proof of the main Theorem}
\vskip 0.3cm
\noindent Assuming we have proved Theorems \ref{theorem LWP} and \ref{theorem1GWP}, we can prove the main Theorem. Firstly, notice that Theorem \ref{theorem LWP} implies  
\begin{equation}
\exists t_{1} \in [0,T],\,\, u(t_{1} ) \in H^2 \cap  B^{\frac{1}{2}}_{2,1}, \quad \hbox{and} \quad \| u(t_{1} ) \|_{H^2}\, \leqslant  \| u_{0} \|_{B^{\frac{1}{2}}_{2,1}}. 
\end{equation}
Moreover, we have a fundamental information on $a(t_{1})$ : 
\begin{equation}
a(t_{1})  \in  B^{\frac{3}{2}}_{2,1} \cap L^{\infty}. 
\end{equation}
\noindent Let us underline that we have, by vertue of Remark \ref{argument 2}, 
\begin{equation}
\int_{\T^3} \frac{1}{1 + a(t_{1}) } \, u(t_{1})  = \int_{\T^3} \frac{1}{1 + a_{0}} \ u_{0} =0.
\end{equation} 
As a consequence, Theorem \ref{theorem1GWP} implies there exists a global solution $(\rho, w)$ of the system (\ref{NSIH1}) associated with  data  $$(\rho, w)_{t=0} \eqdefa \Bigl(\frac{1}{1 +a(t_{1})} ,u(t_{1}) \Bigr).$$ 
\noindent First of all, we adopt the classical point of view : from the solution $(\rho, w)$ of the system (\ref{NSIH1}), we define the solution $(a_{w},w)$ of the system (\ref{NSIH avec a}), given by $$ \rho \eqdefa \frac{1}{1+ a_{w}}.$$
\noindent Therefore, it follows that the solution $(a_{w},w)$ is associated with the data $ (a(t_{1}) ,u(t_{1}))$, which belongs to $B^{\frac{3}{2}}_{2,1} \cap L^{\infty} \times H^2$.\\ 
\noindent The goal is to prove the uniqueness of such a solution, which will come from the following regularity
$$ \forall\,\,  T\geqslant 0,\quad (a_{w}, w) \in \mathcal{C}([0,T],B^{\frac{3}{2}}_{2,1})\, \times\, \mathcal{C}([0,T],B^{\frac{1}{2}}_{2,1})\,\cap\, L^{1}([0,T],B^{\frac{5}{2}}_{2,1})\cdotp$$

\noindent Proving such a regularity on the density function and the velocity field provides us the uniqueness by vertue of local wellposedness Theorem \ref{theorem LWP}. The point is the propagation of the regularity of $a_{w}$.

\subsection{Propagation of the regularity of the density}
\begin{prop}
\label{propagation de la regularity}
Let $T>0$ be a time fixed. Then, $\ds{\forall t \in [0,T], \quad a_{w}(t) \in  B^{\frac{3}{2}}_{2,1}}$.
\end{prop}

\begin{proof}
Applying the frequencies localization operator $\Delta_{q}$ on the transport equation, we get 
\begin{equation}
\partial_{t} \Delta_{q}a_{w} + w\cdotp \nabla \Delta_{q} a_{w}  = - \left[ \Delta_{q},w\cdotp \nabla \right]a_{w} .
\end{equation}
\noindent Taking the $L^2$-inner product with $\Delta_{q}a$, the divergence-free condition implies that
\begin{equation}
\begin{split}
\frac{1}{2}\dfrac{d}{dt} \| \Delta_{q}a_{w} \|^{2}_{L^2}
&\leqslant \|\Delta_{q}a_{w}  \|_{L^2} \, \| \left[ \Delta_{q},w\cdotp \nabla \right]a_{w}  \|_{L^2}.
\end{split} 
\end{equation}
\noindent By vertue of Gronwall's Lemma \ref{gronwall} (given in the appendix), we infer that 
\begin{equation}
\begin{split}
2^{\frac{3q}{2}}\, \| \Delta_{q}a_{w}(t) \|_{L^2} &\leqslant 2^{\frac{3q}{2}}\,\| \Delta_{q}a(t_{1})\|_{L^2} +2^{\frac{3q}{2}}\,\int_{0}^{t} \| \left[ \Delta_{q},w\cdotp \nabla \right]a_{w}  \|_{L^2}\, dt'.\\
\end{split}
\end{equation}
\noindent Therefore, by some classical estimate of the commutator (see Lemma $2.100$ in \cite{BCDbis}), we get 
\begin{equation}
\| a_{w}(t) \|_{B^{\frac{3}{2}}_{2,1}} \leqslant \| a(t_{1}) \|_{B^{\frac{3}{2}}_{2,1}} +C\, \int_{0}^{t} \bigl( \| a_{w}(t') \|_{B^{\frac{3}{2}}_{2,1}} \, \| \nabla w(t') \|_{L^{\infty}} + \| \nabla a_{w}(t') \|_{L^{3}} \, \| \nabla w(t') \|_{B^{\frac{1}{2}}_{6,1}} \bigr)dt'.
\end{equation}
\noindent From the following embedding $B^{\frac{3}{2}}_{2,1} \hookrightarrow B^{1}_{3,1}$ which holds in dimension $3$, 
Gronwall Lemma yields
\begin{equation}
\| a_{w}(t) \|_{B^{\frac{3}{2}}_{2,1}} \leqslant \| a(t_{1})  \|_{B^{\frac{3}{2}}_{2,1}}\,\, \exp{\bigl( C\, \int_{0}^{t} \bigl( \| \nabla w(t') \|_{L^{\infty}} + \| \nabla w(t') \|_{B^{\frac{1}{2}}_{6,1}}\bigr) }\bigr)dt' .
\end{equation}
\noindent It remains to be checked that $\ds{\int_{0}^{t}  \| \nabla w(t') \|_{L^{\infty}}  dt'}$ and $\ds{ \int_{0}^{t} \| \nabla w(t') \|_{B^{\frac{1}{2}}_{6,1}}\, dt'}$ exist for any time. This stems from energy method applying on $w$, thanks to Theorem \ref{theorem1GWP}. 
\noindent Concerning the term $\ds{\int_{0}^{t}  \| \nabla w(t') \|_{L^{\infty}} dt'}$, an interpolation argument gives rise to
\begin{equation}
\begin{split}
\int_{0}^{t}  \| \nabla w(t') \|_{L^{\infty}} dt'  &\leqslant \int_{0}^{t}  \| \nabla w(t') \|^{\frac{1}{4}}_{L^{2}} \,   \| \nabla^2 w(t') \|^{\frac{3}{4}}_{L^{6}} dt'\\
& \leqslant \frac{1}{4} \int_{0}^{t}  \| \nabla w(t') \|_{L^{2}} dt' \, + \frac{3}{4} \int_{0}^{t}  \| \nabla^2 w(t') \|_{L^{6}}dt',
\end{split}
\end{equation}
\noindent and thanks to Hölder's inequality, we get
\begin{equation}
\begin{split}
\int_{0}^{t}  \| \nabla w(t') \|_{L^{\infty}} &\leqslant C\,  t^{\frac{1}{2}}\, \bigl( \| \nabla w(t') \|_{L^{2}_{t}(L^2)} \, +   \| \nabla^2 w(t') \|_{L^{2}_{t}(L^6)} \bigr).
\end{split}
\end{equation}
By vertue of Theorem \ref{theorem1GWP}, $\ds{\| \nabla w \|_{L^{2}_{t}(L^2)} \leqslant C\, \| u(t_{1}) \|_{L^2} }$ and $\ds{\| \nabla^2 w \|_{L^{2}_{t}(L^6)} \leqslant C\,  \| u(t_{1}) \|_{H^2}}$, therefore, 
\begin{equation}
\begin{split}
\int_{0}^{t}  \| \nabla w(t') \|_{L^{\infty}} dt' &\leqslant C\, t^{\frac{1}{2}}\, \| u(t_{1}) \|_{H^2}. 
\end{split}
\end{equation}
\noindent Concerning the term $\ds{\int_{0}^{t} \| \nabla w(t') \|_{B^{\frac{1}{2}}_{6,1}} dt'}$, arguments are similar to the others ones and lead us to
\begin{equation}
\begin{split}
\int_{0}^{t} \| \nabla w(t') \|_{B^{\frac{1}{2}}_{6,1}} dt'  &\leqslant \int_{0}^{t}  \| w(t') \|^{\frac{1}{2}}_{B^{-1}_{6,\infty}} \,   \| \ w(t') \|^{\frac{1}{2}}_{B^{2}_{6,\infty}} dt'\\
& \leqslant \frac{1}{2}\, \int_{0}^{t}  \|  w(t') \|_{B^{-1}_{6,\infty}} dt' \, + \frac{1}{2} \int_{0}^{t}  \|  w(t') \|_{B^{2}_{6,\infty}} dt'.\\
\end{split}
\end{equation}
\noindent Notice we have the following embeddings 
\begin{equation}
L^{2} \hookrightarrow B^{-1}_{6,\infty} \quad \hbox{and} \quad L^{6} \hookrightarrow B^{0}_{6,\infty},
\end{equation}
\noindent from which we infer that (thanks to Thereom \ref{theorem1GWP})
\begin{equation}
\begin{split}
\int_{0}^{t} \| \nabla w \|_{B^{\frac{1}{2}}_{6,1}}
&\leqslant \frac{1}{2}\, \int_{0}^{t}  \|  w \|_{L^2} \, +  \frac{1}{2}\, \int_{0}^{t}  \| \nabla^2 w \|_{L^6}\\
&\leqslant\frac{t}{2}\,  \|  w \|_{L^{\infty}_{t}(L^2)} \, +  \frac{1}{2}\,   t^{\frac{1}{2}} \, \| \nabla^2 w \|_{L^{2}_{t}(L^6)}\\
&\leqslant \frac{1}{2}\,  t\,  \|  u(t_{1}) \|_{L^2} \, +  \frac{1}{2}\,   t^{\frac{1}{2}} \, \| u(t_{1}) \|_{H^{2}}.\\
\end{split}
\end{equation}
\noindent Choosing $t$ small enough such that $t \leqslant t^{\frac{1}{2}}$, we get
$$\int_{0}^{t} \| \nabla w(t') \|_{B^{\frac{1}{2}}_{6,1}} dt' \leqslant C\,   t^{\frac{1}{2}} \, \| u(t_{1}) \|_{H^{2}}.$$

\noindent This yields to the desired estimate
\begin{equation}
\label{density function}
\| a_{w}(t) \|_{B^{\frac{3}{2}}_{2,1}} \leqslant \| a(t_{1})  \|_{B^{\frac{3}{2}}_{2,1}}\,\, \exp{\bigl( C\, t^{\frac{1}{2}} \, \| u(t_{1}) \|_{H^{2}}\bigr)}.
\end{equation}
\noindent This concludes the proof on the propagation of the regularity on the density function.

\subsection{Regularity of the velocity field}
\noindent Holding the regularity on the density, we are allowed to apply Lemma \ref{lemma general avec linfini}, which gives rise to the following estimate, available, for any $t \in [0,T]$, where $T$ is a fixed finite time.
\begin{equation}
\label{estimate on velocity w}
\|w \|_{L^{\infty}_{t}(B^{\frac{1}{2}}_{2,1})} + \|w \|_{L^{1}_{t}(B^{\frac{5}{2}}_{2,1})} + \|\nabla \Pi \|_{L^{1}_{t}(B^{\frac{1}{2}}_{2,1})} \leqslant C\, \| u(t_{1}) \|_{B^{\frac{1}{2}}_{2,1}} + C\, \int_{0}^{t} \, \bigl( \|\nabla w(t') \|_{L^{\infty}}\, + \, W(t')  \bigr) \|w(t') \|_{B^{\frac{1}{2}}_{2,1}} dt', 
\end{equation}  
\noindent where $$ W(t') \eqdefa 2^{2m}\, \| a_{w} \|^{2}_{L^{\infty}_{t'}(L^{\infty})} \, +\, 2^{8m}\, \| a_{w}  \|^{4}_{L^{\infty}_{t'}(L^{2})}\, \bigl( 1 +\, \|w\|^{4}_{L^{\infty}_{t'}(B^{\frac{1}{2}}_{2,1})}\, + \| a_{w}  \|^{4}_{L^{\infty}_{t'}(L^{\infty})}   \bigr).$$  
\noindent We deduce from this estimate, by Gronwall Lemma, 
\begin{equation}
\|w \|_{L^{\infty}_{t}(B^{\frac{1}{2}}_{2,1})} + \|w \|_{L^{1}_{t}(B^{\frac{5}{2}}_{2,1})} + \|\nabla \Pi \|_{L^{1}_{t}(B^{\frac{1}{2}}_{2,1})} \leqslant C\, \| u(t_{1}) \|_{B^{\frac{1}{2}}_{2,1}} \, \exp{\bigl(  \|\nabla w \|_{L^{1}_{t}(L^{\infty})} \, + t\,  W(t)\,    \bigr)}.
\end{equation}
\noindent Concerning the term $W(t)$, on the one hand, by the transport equation, we get immediately $$ \|a_{w} (t,\cdotp) \|_{L^{2}} =  \|a_{w}(0,\cdotp) \|_{L^{2}},$$
which is bounded by $\|a_{w}(0) \|_{B^{\frac{3}{2}}_{2,1}}$, since spaces are inhomogeneous. One the other hand, by an interpolation argument, one has
\begin{equation*}
\begin{split}
\|w\|_{B^{\frac{1}{2}}_{2,1}}\, &\leqslant \|w\|^{\frac{1}{2}}_{B^{0}_{2,\infty}}\, \|w\|^{\frac{1}{2}}_{B^{1}_{2,\infty}}\\
&\leqslant \|w\|^{\frac{1}{2}}_{L^2}\, \|w\|^{\frac{1}{2}}_{H^{1}}.\\
\end{split}
\end{equation*}
\noindent It follows that, by vertue of Theorem \ref{theorem1GWP},
\begin{equation}
\label{norme b 1/2 de w, interpolation faite}
\begin{split}
\|w\|_{L^{\infty}_{t}(B^{\frac{1}{2}}_{2,1})}\,  &\leqslant \|w\|^{\frac{1}{2}}_{L^{\infty}_{t}(L^{2})}\, \|w\|^{\frac{1}{2}}_{L^{\infty}_{t}(H^{1})}\\
&\leqslant \| u(t_{1})  \|^{\frac{1}{2}}_{L^{2}}\,  \bigl( \| u(t_{1})  \|^{\frac{1}{2}}_{L^{2}}\, +  \| \nabla u(t_{1})  \|^{\frac{1}{2}}_{L^{2}} \bigr). 
\end{split}
\end{equation}
\noindent It results from these simple computations that the factor $W(t)$ is bounded by
$$\forall t \in [0,T], \,\, W(t) \leqslant C\,  \| u(t_{1})  \|_{H^{2}}.$$
\noindent As it has been already noticed, the term $\|\nabla w \|_{L^{1}_{t}(L^{\infty})}$ satisfies
\begin{equation}
\begin{split}
\int_{0}^{t}  \| \nabla w(t') \|_{L^{\infty}} dt' \, &\leqslant C\, t^{\frac{1}{2}}\, \| u(t_{1}) \|_{H^2}. 
\end{split}
\end{equation}
\noindent It results from all of this, that for any $t \in [0,T]$, we have
\begin{equation}
\|w \|_{L^{\infty}_{t}(B^{\frac{1}{2}}_{2,1})} + \|w \|_{L^{1}_{t}(B^{\frac{5}{2}}_{2,1})} + \|\nabla \Pi \|_{L^{1}_{t}(B^{\frac{1}{2}}_{2,1})} \leqslant C\, \| u(t_{1}) \|_{B^{\frac{1}{2}}_{2,1}} \, \exp{\bigl(  C\,  t^{\frac{1}{2}}\, \| u(t_{1}) \|_{H^2}  \bigr)}. 
\end{equation}
\noindent Combining with the estimate on the density function (\ref{density function}), we get $t \in [0,T]$, for a fixed time $T>0$ 
\begin{equation}
\| a_{w} \|_{L^{\infty}_{t}(B^{\frac{3}{2}}_{2,1})}   + \|w \|_{L^{\infty}_{t}(B^{\frac{1}{2}}_{2,1})} + \|w \|_{L^{1}_{t}(B^{\frac{5}{2}}_{2,1})} + \|\nabla \Pi \|_{L^{1}_{t}(B^{\frac{1}{2}}_{2,1})} \leqslant C\, \| u(t_{1}) \|_{B^{\frac{1}{2}}_{2,1}} \, \exp{\bigl(  C\,  t^{\frac{1}{2}}\, \| u(t_{1}) \|_{H^2}  \bigr)}. 
\end{equation}
\noindent This ends up the proof of Theorem \ref{main theorem}. 
\vskip 0.3cm

\end{proof}

\section{Proof of the local wellposedness part of the main theorem }
\vskip 0.2cm
\noindent This section is devoted to the proof of Theorem \ref{theorem LWP}. We give only the proof of the existence part of the theorem, since the uniqueness part has been already proved in \cite{AGZ}. We only mention the start point of the uniqueness proof.
\subsection{Existence part} 
\noindent The existence proof can be achieved by a regularization process (e.g Fridriech method). The idea is classical : we build smooth approximate solutions, perform uniform estimates on them. A compactness argument leads us to the proof of the existence of a solution of \ref{NSIH avec a}. We skip this part and provide some a priori estimates for smooth enough solution $(a,u)$.\\
\vskip 0.2cm
\noindent Let us start by proving the estimate (\ref{density}) on the density. Applying the frequencies localization operator~$\Delta_{q}$ on the transport equation, we get 
\begin{equation*}
\partial_{t} \Delta_{q}a + u\cdotp \nabla \Delta_{q}a = - \left[ \Delta_{q},u\cdotp \nabla \right]a.
\end{equation*}
\noindent Taking the $L^2$-inner product with $\Delta_{q}a$, the divergence-free condition implies that
\begin{equation*}
\begin{split}
\frac{1}{2}\dfrac{d}{dt} \| \Delta_{q}a\|^{2}_{L^2} &= -\bigl( \left[ \Delta_{q},u\cdotp \nabla \right]a \, \vert \, \Delta_{q}a  \bigr)_{L^2}\\
&\leqslant \|\Delta_{q}a \|_{L^2} \, \| \left[ \Delta_{q},u\cdotp \nabla \right]a \|_{L^2}.
\end{split} 
\end{equation*}
\noindent By vertue of Gronwall's Lemma \ref{gronwall} (given in the appendix), we infer that 
\begin{equation*}
\begin{split}
2^{\frac{3q}{2}}\, \| \Delta_{q}a\|_{L^2} &\leqslant 2^{\frac{3q}{2}}\,\| \Delta_{q}a_{0}\|_{L^2} +2^{\frac{3q}{2}}\, \int_{0}^{t} \| \left[ \Delta_{q},u\cdotp \nabla \right]a \|_{L^2}\, dt'.\\
\end{split}
\end{equation*}
\noindent A classical commutator estimate (see for instance Lemma $2.100$ in \cite{BCDbis}) shows there exists a sequence $(c_{q})$ belonging to $\ell^{1}(\Z)$ such that 
$$ 2^{\frac{3q}{2}}\, \| \left[ \Delta_{q},u\cdotp \nabla \right]a \|_{L^2}\, \leqslant c_{q}\, \| a \|_{B^{\frac{3}{2}}_{2,1}} \, \| u \|_{B^{\frac{5}{2}}_{2,1}},$$
\noindent and therefore, 
\begin{equation*}
2^{\frac{3q}{2}}\, \int_{0}^{t} \| \left[ \Delta_{q},u\cdotp \nabla \right]a \|_{L^2}\, dt' \leqslant \sup_{t}\, c_{q}(t) \, \int_{0}^{t} \| a(t') \|_{B^{\frac{3}{2}}_{2,1}} \, \| u(t') \|_{B^{\frac{5}{2}}_{2,1}} dt'.
\end{equation*}
By summing on $q \in \Z$, we get
\begin{equation*}
\| a\|_{B^{\frac{3}{2}}_{2,1}} \leqslant \| a_{0}\|_{B^{\frac{3}{2}}_{2,1}} +C\, \int_{0}^{t} \| a(t') \|_{B^{\frac{3}{2}}_{2,1}} \, \| u(t') \|_{B^{\frac{5}{2}}_{2,1}}\, dt'. 
\end{equation*}
\noindent The classical Gronwall's Lemma yields the proof of (\ref{density}).
\vskip 0.3cm

\noindent Let us prove estimate (\ref{velocity}) on the velocity. Actually, we prove Lemma \ref{lemma general avec linfini}, which is a bit more general than we want to get. 

\vskip 0.2cm
\noindent \textit{Proof of Lemma \ref{lemma general avec linfini}}. \\We may rewrite the system (\ref{NSIH avec a}), after decomposing $(1+a)$ into\, $\ds{(1+ S_{m}a)\, +\, (a- S_{m}a)}$.  
\begin{equation}
\label{NSIH avec S_{m}a} 
\begin{split} 
 \partial_{t} u + u\cdot\nabla{u} -(1+S_{m}a)\,\Delta{u} + (1+S_{m}a)\,\nabla{\Pi} &=(a- S_{m}a)(\Delta u - \nabla \Pi) \\
\end{split}
\end{equation}
\noindent Notice that $\ds{(1+S_{m}a)\,\nabla{\Pi} \,=\, \nabla \bigl((1+S_{m}a)\,\Pi\bigr) -  \Pi\, \nabla S_{m}a\, }$, which implies
\begin{equation*}
 \partial_{t} u + u\cdot\nabla{u} -(1+S_{m}a)\,\Delta{u} + \nabla \bigl((1+S_{m}a)\,\Pi\bigr) =(a- S_{m}a)(\Delta u - \nabla \Pi) \, + \, \Pi\, \nabla S_{m}a.
\end{equation*}
\vskip 0.2cm
\noindent Let us introduce the notation $\ds{E_{m} \eqdefa (a- S_{m}a)(\Delta u - \nabla \Pi)}$. We reduce the problem to the system below 
\begin{equation}
\label{NSIH avec S_{m}a final} \left \lbrace \begin {array}{ccc}  \partial_{t} u + u\cdot\nabla{u} -(1+S_{m}a)\,\Delta{u} + \nabla \bigl((1+S_{m}a)\,\Pi\bigr) &=& E_{m} \, + \, \Pi\, \nabla S_{m}a.  \\
 \dive u&=&0\\
 (a,u)_{|t=0}&=&(a_{0},u_{0}),\\ \end{array}
\right.
\end{equation}
\vskip 0.5cm
\noindent \textit{Step 1: Frequency localization.}\\
\noindent Applying the operator $\Delta_{q}$ in (\ref{NSIH avec S_{m}a final}), we localize the velocity in a ring, with a size $2^q$, and we get 
\begin{equation*}
\partial_{t}\Delta_{q}u + \Delta_{q}(u\cdotp \nabla u) - \Delta_{q}\bigl((1+S_{m}a)\,\Delta{u}  \bigr) +  \Delta_{q}\bigl(\nabla \bigl((1+S_{m}a)\,\Pi\bigr) \bigr) \,=\, \Delta_{q}E_{m} + \Delta_{q}(\Pi\, \nabla S_{m}a). 
\end{equation*} 
\noindent By definition of the commutator $\ds{\Delta_{q}(u\cdotp \nabla u) \eqdefa u\cdotp \nabla \Delta_{q}u \, + \, \left[ \Delta_{q} , u\cdotp \nabla \right]u }$, this gives
\begin{equation*}
\begin{split}
\partial_{t}\Delta_{q}u + u\cdotp \nabla \Delta_{q}u - \Delta_{q}\bigl((1+S_{m}a)\,\Delta{u}  \bigr) +  \Delta_{q}\bigl(\nabla \bigl((1+S_{m}a)\,\Pi\bigr) \bigr) &=\, -\left[ \Delta_{q} , u\cdotp \nabla \right]u   \,+ \Delta_{q}E_{m}\\ &\qquad + \Delta_{q}(\Pi\, \nabla S_{m}a). 
\end{split}
\end{equation*}
\noindent In particular, a simple computation gives $$ - \Delta_{q}\bigl((1+S_{m}a)\,\Delta{u}  \bigr) = -\dive\bigl((1+S_{m}a)\,\Delta_{q} \nabla u  \bigr) - \dive\bigl(\left[ \Delta_{q}, S_{m}a \right]\nabla u \bigr)+ \Delta_{q}\bigl( \nabla S_{m}a\, \nabla u \bigr).$$
\noindent As a consequence, we get
\begin{equation}
\label{estimate with localization}
\begin{split}
&\partial_{t}\Delta_{q}u + u\cdotp \nabla \Delta_{q}u -\dive\Bigl((1+S_{m}a)\,\Delta_{q} \nabla u\,  +  \Delta_{q}\Bigl(\nabla \bigl((1+S_{m}a)\,\Pi\bigr) \Bigr) =\, -\left[ \Delta_{q} , u\cdotp \nabla \right]u   \,+ \Delta_{q}E_{m}\\ &\qquad \qquad \qquad + \Delta_{q}(\Pi\, \nabla S_{m}a) + \dive \bigl( \left[ \Delta_{q}, S_{m}a \right]\nabla u \bigr) -\, \Delta_{q}\bigl( \nabla S_{m}a\, \nabla u \bigr).
\end{split}
\end{equation}
\vskip 0.2cm
\noindent Let us take the $L^2$ inner product with $\Delta_{q} u$ in the above equation (\ref{estimate with localization}). Because of the divergence free condition, we have
$$\bigl( u\cdotp \nabla \Delta_{q}u \,|\, \Delta_{q}u \bigr)_{L^2} =0 \quad \hbox{and} \quad \bigl( \Delta_{q}\bigl(\nabla \bigl((1+S_{m}a)\,\Pi\bigr) \bigr) \,|\, \Delta_{q} u \bigr)_{L^2} =0.$$ As a result, 
\begin{equation*}
\begin{split}
\frac{1}{2}\dfrac{d}{dt} \|\Delta_{q} u \|^{2}_{L^2} + \int_{\T^3} (1+S_{m}a)\,|\Delta_{q} \nabla u |^2\, dx &\leqslant \| \Delta_{q} u \|_{L^2} \, \Bigl( \|  \left[ \Delta_{q} , u\cdotp \nabla \right]u \|_{L^2} + \| \Delta_{q} E_{m} \|_{L^2} + \| \Delta_{q}(\Pi\, \nabla S_{m}a) \|_{L^2}\\ &\qquad + 2^{q}\, \| \left[ \Delta_{q}, S_{m}a \right]\nabla u \|_{L^2} +  \| \Delta_{q}\bigl( \nabla S_{m}a\, \nabla u \bigr) \|_{L^2}    \Bigr)
\end{split}
\end{equation*} 
\noindent Let us point that $\ds{1+S_{m}a = 1 + a + S_{m}a - a}$. As we assume that $\ds{S_{m}a - a}$ is small enough in norm $L^{\infty}_{t}(B^{\frac{3}{2}}_{2,1})$, it follows that
$$ 1+S_{m}a \geqslant \frac{b}{2},$$
\noindent which along with Lemma \ref{derivée dans les besov}, ensures that
\begin{equation*}
\begin{split}
\frac{1}{2}\dfrac{d}{dt} \|\Delta_{q} u \|^{2}_{L^2} + \frac{b}{2}\, \,2^{2q}\,  \|\Delta_{q} u \|^2_{L^2} &\leqslant \| \Delta_{q} u \|_{L^2} \, \Bigl( \|  \left[ \Delta_{q} , u\cdotp \nabla \right]u \|_{L^2} + \| \Delta_{q} E_{m} \|_{L^2} + \| \Delta_{q}(\Pi\, \nabla S_{m}a) \|_{L^2}\\ &\qquad + 2^{q}\, \| \left[ \Delta_{q}, S_{m}a \right]\nabla u \|_{L^2} +  \| \Delta_{q}\bigl( \nabla S_{m}a\, \nabla u \bigr) \|_{L^2}    \Bigr).
\end{split}
\end{equation*} 
\noindent Applying a Gronwall's argument, we get 
\begin{equation*}
\begin{split}
\dfrac{d}{dt} \|\Delta_{q} u \|_{L^2} + \frac{b}{2}\, \,2^{2q}\,  \|\Delta_{q} u \|_{L^2} &\leqslant  \|  \left[ \Delta_{q} , u\cdotp \nabla \right]u \|_{L^2} + \| \Delta_{q} E_{m} \|_{L^2} + \| \Delta_{q}(\Pi\, \nabla S_{m}a) \|_{L^2}\\ &\qquad + 2^{q}\, \| \left[ \Delta_{q}, S_{m}a \right]\nabla u \|_{L^2} +  \| \Delta_{q}\bigl( \nabla S_{m}a\, \nabla u \bigr) \|_{L^2}.
\end{split}
\end{equation*}
\noindent An integration in time yields
\begin{equation*}
\begin{split}
2^{\frac{q}{2}}\, \|\Delta_{q} u \|_{L^2} +\,C\,b\,2^{\frac{5\,q}{2}} \,\int_{0}^{t} \|\Delta_{q} u \|_{L^2}\,dt' &\leqslant\,\, 2^{\frac{q}{2}}\,\|\Delta_{q} u_{0} \|_{L^2}\,  + \int_{0}^{t}  2^{\frac{q}{2}}\,  \|  \left[ \Delta_{q} , u\cdotp \nabla \right]u \|_{L^2}dt'\,\\ &\qquad + \int_{0}^{t}  2^{\frac{q}{2}}\, \| \Delta_{q} E_{m} \|_{L^2}dt'\, + \int_{0}^{t}  2^{\frac{q}{2}}\, \| \Delta_{q}(\Pi\, \nabla S_{m}a) \|_{L^2}\,dt'\\  &\qquad +\int_{0}^{t}  2^{\frac{3q}{2}}\,  \| \left[ \Delta_{q}, S_{m}a \right]\nabla u \|_{L^2} dt'\, + \int_{0}^{t}  2^{\frac{q}{2}}\,  \| \Delta_{q}\bigl( \nabla S_{m}a\, \nabla u \bigr) \|_{L^2} dt'.
\end{split}
\end{equation*} 
\noindent Taking the supremium in time and then summing on $q \in \Z$ provides us the norm $\ds{\| u \|_{L^{\infty}_{t}(B^{\frac{1}{2}}_{2,1})}}$ and thus
\begin{equation}
\label{estimate2}
\begin{split}
\| u \|_{L^{\infty}_{t}(B^{\frac{1}{2}}_{2,1})}\, +\,C\,b\, \|  u \|_{L^{1}_{t}(B^{\frac{5}{2}}_{2,1})}\, &\leqslant\,\, \| u_{0} \|_{B^{\frac{1}{2}}_{2,1}}\, +\, \| E_{m} \|_{L^{1}_{t}(B^{\frac{1}{2}}_{2,1})}\, + \, \| \Pi\, \nabla S_{m}a \|_{L^{1}_{t}(B^{\frac{1}{2}}_{2,1})}\\     &\qquad + \sum_{q \in \Z} 2^{\frac{q}{2}}\,  \|  \left[ \Delta_{q} , u\cdotp \nabla \right]u \|_{L^{1}_{t}(L^2)}\\  &\qquad + \sum_{q \in \Z}  2^{\frac{3q}{2}}\,  \| \left[ \Delta_{q}, S_{m}a \right]\nabla u \|_{L^{1}_{t}(L^2)} +  \| \nabla S_{m}a\, \nabla u  \|_{L^{1}_{t}(B^{\frac{1}{2}}_{2,1})}.
\end{split}
\end{equation}
\medbreak
\noindent \textit{Step 2: Estimate of each term in the right-hand-side of the above inequality.}\\
\vskip 0.2cm
\noindent $\star$ Estimate of $\| E_{m} \|_{L^{1}_{t}(B^{\frac{1}{2}}_{2,1})}$\\
\noindent Product laws in Besov spaces (cf Lemma \ref{lemma product law} in Appendix) yield
\begin{equation}
\begin{split}
\| E_{m} \|_{L^{1}_{t}(B^{\frac{1}{2}}_{2,1})} &\leqslant C\, \| a - S_{m}a \|_{L^{\infty}_{t}(B^{\frac{3}{2}}_{2,1})}\, \| \Delta u - \nabla \Pi \|_{L^{1}_{t}(B^{\frac{1}{2}}_{2,1})}\\
&\leqslant C\, \| a - S_{m}a \|_{L^{\infty}_{t}(B^{\frac{3}{2}}_{2,1})}\, \Bigl( \|  \nabla \Pi \|_{L^{1}_{t}(B^{\frac{1}{2}}_{2,1})} + \|  u \|_{L^{1}_{t}(B^{\frac{5}{2}}_{2,1})}    \Bigr). 
\end{split} 
\end{equation}

\medbreak
\noindent $\star$ Estimate of $\| \Pi\, \nabla S_{m}a \|_{L^{1}_{t}(B^{\frac{1}{2}}_{2,1})}$. \\Concerning the pressure term, as it is defined up to a constant, we can assume it is mean free. Same remark holds for the term $\| \nabla S_{m}a \|_{B^{1}_{2,2}}$, since obviously the term $  \nabla S_{m}a $ is mean free. In this way, the norms $\| \cdotp\|_{B^{1}_{2,2}} $  and $ \| \cdotp \|_{\dot{B}^{1}_{2,2}}$ are equivalent. By vertue of paradifferential calculus in inhomogeneous Besov norm, we get
\begin{equation}
\begin{split}
\| \Pi\, \nabla S_{m}a \|_{B^{\frac{1}{2}}_{2,1}} &\leqslant C\, \| \Pi\ \|_{B^{1}_{2,2}}  \, \| \nabla S_{m}a \|_{B^{1}_{2,2}}\\
&\leqslant C\,  \| \Pi\ \|_{\dot{B}^{1}_{2,2}}  \, \| \nabla S_{m}a \|_{\dot{B}^{1}_{2,2}}\\
&= C\, \| \nabla \Pi\ \|_{L^2} \, \| \nabla S_{m}a \|_{\dot{H}^{1}},
\end{split} 
\end{equation}
\noindent which leads to 
$$\| \Pi\, \nabla S_{m}a \|_{L^{1}_{t}(B^{\frac{1}{2}}_{2,1})} \leqslant C\, \| \nabla \Pi\ \|_{L^{1}_{t}(L^2)}\,  \| \nabla S_{m}a \|_{L^{\infty}_{t}(\dot{H}^{1})}. $$

\medbreak
\noindent $\star$ Estimate of $\ds{\| \nabla S_{m}a\, \nabla u  \|_{L^{1}_{t}(B^{\frac{1}{2}}_{2,1})}}$. Above arguments still provide\\
\begin{equation}
\begin{split}
\| \nabla S_{m}a\, \nabla u  \|_{B^{\frac{1}{2}}_{2,1}} &\leqslant C\, \| \nabla S_{m}a \|_{B^{1}_{2,2}}  \, \| \nabla u \|_{B^{1}_{2,2}}\\
&\leqslant C\,  \| \nabla u \|_{\dot{H}^1}  \, \| \nabla S_{m}a \|_{\dot{H}^{1}}.\\
\end{split} 
\end{equation}
Therefore, we deduce that 
$$\| \nabla S_{m}a\, \nabla u  \|_{L^{1}_{t}(B^{\frac{1}{2}}_{2,1})} \leqslant C\, \| u \|_{L^{1}_{t}(\dot{H}^2)}  \, \| \nabla S_{m}a \|_{L^{\infty}_{t}(\dot{H}^{1})}.$$

\medbreak
\noindent $\star$ Estimate of $\ds{\sum_{q \in \Z} 2^{\frac{q}{2}}\,  \|  \left[ \Delta_{q} , u\cdotp \nabla \right]u \|_{L^{1}_{t}(L^2)}}$.
\noindent By vertue of commutator estimate, we infer that 
\begin{equation*}
\|  \left[ \Delta_{q} , u\cdotp \nabla \right]u \|_{L^2} \leqslant C\, d_{q}\, 2^{-\frac{q}{2}}\, \| \nabla u \|_{B^{\frac{3}{2}}_{2,1}}\, \| u \|_{B^{\frac{1}{2}}_{2,1}}.  
\end{equation*} 
Therefore, we deduce that 
$$ \sum_{q \in \Z} 2^{\frac{q}{2}}\,  \|  \left[ \Delta_{q} , u\cdotp \nabla \right]u \|_{L^{1}_{t}(L^2)} \leqslant C\, \int_{0}^{t} \| \nabla u(t') \|_{L^{\infty}}\, \| u(t') \|_{B^{\frac{1}{2}}_{2,1}} \,dt'.$$

\medbreak
\noindent $\star$ Estimate of $\ds{\sum_{q \in \Z}  2^{\frac{3q}{2}}\,  \| \left[ \Delta_{q}, S_{m}a \right]\nabla u \|_{L^{1}_{t}(L^2)}}$.
\noindent We can prove the estimate below (see Lemma \ref{derivée dans les besov}) 
\begin{equation}
\sum_{q \in \Z}  2^{\frac{3q}{2}}\,  \| \left[ \Delta_{q}, S_{m}a \right]\nabla u \|_{L^{1}_{t}(L^2)} \leqslant\, C\, 2^{m}\, \| a\|_{L^{\infty}_{t}(L^{\infty})} \, \| u \|_{L^{1}_{t}(B^{\frac{3}{2}}_{2,1})}\, + \,2^{2m}\, \| a\|_{L^{\infty}_{t}(L^{2})} \, \| u \|_{L^{1}_{t}(\dot{H}^{2})}.
\end{equation}
\medbreak
\noindent Plugging all the above estimates in (\ref{estimate2}), we finally get 
\begin{equation}
\label{estimate3}
\begin{split}
\| u \|_{L^{\infty}_{t}(B^{\frac{1}{2}}_{2,1})}\, +\,C\,b\, \|  u \|_{L^{1}_{t}(B^{\frac{5}{2}}_{2,1})}\, &\leqslant\,\, \| u_{0} \|_{B^{\frac{1}{2}}_{2,1}}\, + \| a - S_{m}a \|_{L^{\infty}_{t}(B^{\frac{3}{2}}_{2,1})}\, \Bigl( \|  \nabla \Pi \|_{L^{1}_{t}(B^{\frac{1}{2}}_{2,1})} + \|  u \|_{L^{1}_{t}(B^{\frac{5}{2}}_{2,1})}    \Bigr)\\
&\qquad + \int_{0}^{t} \| \nabla u(t') \|_{L^{\infty}}\, \| u(t') \|_{B^{\frac{1}{2}}_{2,1}}\, dt'\, + \, 2^{m}\, \| a\|_{L^{\infty}_{t}(L^{\infty})} \, \| u \|_{L^{1}_{t}(B^{\frac{3}{2}}_{2,1})}\\
&\qquad + 2^{2m+1}\, \| a\|_{L^{\infty}_{t}(L^{2})} \, \Bigl( \| u \|_{L^{1}_{t}(\dot{H}^{2})} +\, \|  \nabla \Pi \|_{L^{1}_{t}(L^2)}  \Bigr),
\end{split}
\end{equation} 
\noindent where we have used
$$ \| \nabla S_{m}a \|_{L^{\infty}_{t}(\dot{H}^{1})} =  \| \nabla^2 S_{m}a \|_{L^{\infty}_{t}(L^{2})} \leqslant 2^{2m}\, \| a\|_{L^{\infty}_{t}(L^{2})}.$$
\vskip 0.3cm

\noindent \textit{Step 3: Estimate of $\|  \nabla \Pi \|_{L^{1}_{t}(B^{\frac{1}{2}}_{2,1})}$}.\\
\noindent We take the divergence operator in (\ref{NSIH avec S_{m}a}) and thus 
\begin{equation*}
\dive\bigl( (1+S_{m}a)\nabla \Pi \bigr) = -\dive(u\cdotp \nabla u) + \Delta u\cdotp\nabla S_{m}a\, + \dive \Bigl( (S_{m}a-a)(\nabla \Pi - \Delta u)  \Bigr). 
\end{equation*}
\noindent Applying  the operator $\Delta_{q}$ and taking the $L^2$ inner product with $\Delta_{q} \Pi$ yield
\begin{equation*}
\begin{split}
\bigl( \Delta_{q}\Bigl((1+S_{m}a)\nabla \Pi \Bigr) | \Delta_{q} \Pi    \bigr)_{L^2} &=  \bigl( \Delta_{q}\bigl( u\cdotp \nabla u \bigr) | \Delta_{q} \nabla \Pi    \bigr)_{L^2} + \bigl( \Delta_{q}\bigl( \Delta u\cdotp\nabla S_{m}a\, \bigr) | \Delta_{q} \Pi    \bigr)_{L^2}\\ 
&+ \bigl( \Delta_{q}\bigl( (S_{m}a-a)\nabla \Pi  \bigr) | \Delta_{q} \nabla \Pi    \bigr)_{L^2} - \bigl( \Delta_{q}\bigl( (S_{m}a-a)\Delta u  \bigr) | \Delta_{q} \nabla \Pi    \bigr)_{L^2}.  
\end{split}
\end{equation*}
\noindent In particular, the left-hand-side can be rewritten and bounded from below as follows
\begin{equation*}
\begin{split}
\bigl( \Delta_{q}\Bigl((1+S_{m}a)\nabla \Pi \Bigr) | \Delta_{q} \nabla \Pi    \bigr)_{L^2} &= \bigl( \Delta_{q}\nabla \Pi  | \Delta_{q} \nabla \Pi    \bigr)_{L^2} + \bigl( \left[ \Delta_{q}, S_{m}a   \right] \nabla \Pi  | \Delta_{q} \nabla \Pi    \bigr)_{L^2}\\
&\quad + \bigl( S_{m}a\, \Delta_{q}\nabla \Pi  | \Delta_{q} \nabla \Pi  \bigr)_{L^2}\\
&= \bigl( (1+S_{m}a)\, \Delta_{q}\nabla \Pi  | \Delta_{q} \nabla \Pi  \bigr)_{L^2}  + \bigl( \left[ \Delta_{q}, S_{m}a   \right] \nabla \Pi  | \Delta_{q} \nabla \Pi    \bigr)_{L^2}\\
&\geqslant b\, \| \Delta_{q} \nabla \Pi \|^{2}_{L^2} + \bigl( \left[ \Delta_{q}, S_{m}a   \right] \nabla \Pi  | \Delta_{q} \nabla \Pi    \bigr)_{L^2}. 
\end{split}  
\end{equation*}
\noindent It follows
\begin{equation}
\begin{split}
b\, \| \Delta_{q} \nabla \Pi \|^{2}_{L^2} &\leqslant  \| \Delta_{q} \nabla \Pi \|_{L^2} \Bigl( \|\Delta_{q}\bigl( u\cdotp \nabla u \bigr) \|_{L^2} + \| \Delta_{q}\bigl( (S_{m}a-a)\nabla \Pi  \bigr)\|_{L^2}\\ &\qquad + \| \Delta_{q}\bigl( (S_{m}a-a)\Delta u  \bigr)\|_{L^2} + \|  \left[ \Delta_{q}, S_{m}a   \right] \nabla \Pi  \|_{L^2} \Bigr) \\ &+ \| \Delta_{q} \Pi \|_{L^2} \, \| \Delta_{q}\bigl( \Delta u \cdotp \nabla S_{m}a   \bigr) \|_{L^2}.  
\end{split}  
\end{equation}
\noindent In particular, Lemma \ref{derivée dans les besov} provides the inequality below 
$$ \| \Delta_{q} \Pi \|_{L^2} \lesssim 2^{-q}\, \|\Delta_{q} \nabla \Pi\|_{L^2},$$
\noindent which gives rise to
\begin{equation}
\begin{split}
b\, \| \Delta_{q} \nabla \Pi \|_{L^2} &\leqslant   \|\Delta_{q}\bigl( u\cdotp \nabla u \bigr) \|_{L^2} + \| \Delta_{q}\bigl( (S_{m}a-a)\nabla \Pi  \bigr)\|_{L^2}\, + \| \Delta_{q}\bigl( (S_{m}a-a)\Delta u  \bigr)\|_{L^2}\\ &+ \|  \left[ \Delta_{q}, S_{m}a   \right] \nabla \Pi  \|_{L^2}\,  + 2^{-q}\,  \| \Delta_{q}\bigl( \Delta u \cdotp \nabla S_{m}a   \bigr) \|_{L^2}.  
\end{split}  
\end{equation} 
\noindent Multiplying by $2^{\frac{q}{2}}$ and summing on $q \in \Z$, we have
\begin{equation*}
\begin{split}
b\, \| \nabla \Pi \|_{B^{\frac{3}{2}}_{2,1}} &\lesssim \, \| u\cdotp \nabla u \|_{B^{\frac{1}{2}}_{2,1}}                    + \, \| (S_{m}a-a)\nabla \Pi \|_{B^{\frac{1}{2}}_{2,1}} + \| (S_{m}a-a)\Delta u \|_{B^{\frac{1}{2}}_{2,1}}\\
&\quad + \| \Delta u \cdotp \nabla S_{m}a  \|_{B^{-\frac{1}{2}}_{2,1}} + \sum_{q \in \Z} 2^{\frac{q}{2}} \| \left[ \Delta_{q}, S_{m}a   \right] \nabla \Pi \|_{L^2}.                     
\end{split}  
\end{equation*}
\noindent Notice that 
$$ \| \Delta u \cdotp \nabla S_{m}a  \|_{B^{-\frac{1}{2}}_{2,1}} \leqslant C\, \| \nabla S_{m}a \|_{\dot{H}^1}\, \| \Delta u\|_{L^2}.$$
\noindent On the one hand, product laws in Besov spaces (cf Lemma \ref{lemma product law}) give 
$$ \| u\cdotp \nabla u \|_{B^{\frac{1}{2}}_{2,1}} \leqslant \| u \|_{B^{\frac{1}{2}}_{2,1}}\, \| \nabla u \|_{L^{\infty}}\,.$$
$$ \| (S_{m}a-a)\Delta u \|_{B^{\frac{1}{2}}_{2,1}} \leqslant C\, \| (S_{m}a-a) \|_{B^{\frac{3}{2}}_{2,1}} \, \|\Delta u \|_{B^{\frac{1}{2}}_{2,1}}.$$
$$ \| (S_{m}a-a)\nabla \Pi \|_{B^{\frac{1}{2}}_{2,1}} \leqslant C\, \| (S_{m}a-a) \|_{B^{\frac{3}{2}}_{2,1}} \, \|\nabla \Pi \|_{B^{\frac{1}{2}}_{2,1}}.$$
\noindent On the other hand, a classical commutator estimate yields
$$ \sum_{q \in \Z} 2^{\frac{q}{2}} \| \left[ \Delta_{q}, S_{m}a   \right] \nabla \Pi \|_{L^2} \leqslant C\, \| \nabla S_{m}a \|_{\dot{H}^1}\, \|\nabla \Pi\|_{L^2}.$$
\noindent As a result, previous estimates imply  
\begin{equation}
\begin{split}
b\, \|\nabla \Pi \|_{L^{1}_{t}(B^{\frac{1}{2}}_{2,1})} &\lesssim \int_{0}^{t} \| u(t') \|_{B^{\frac{1}{2}}_{2,1}}\, \| \nabla u(t') \|_{L^{\infty}}\, \, dt' +\,  \| (S_{m}a-a) \|_{L^{\infty}_{t}(B^{\frac{3}{2}}_{2,1})} \, \|\Delta u \|_{L^{1}_{t}(B^{\frac{1}{2}}_{2,1})}\\ &+ \, \| (S_{m}a-a) \|_{L^{\infty}_{t}(B^{\frac{3}{2}}_{2,1})} \, \|\nabla \Pi \|_{L^{1}_{t}(B^{\frac{1}{2}}_{2,1})} 
+ \| \nabla S_{m}a \|_{L^{\infty}_{t}(\dot{H}^1)}\, \Bigl( \|\nabla \Pi\|_{L^{1}_{t}(L^2)} + \| \Delta u\|_{L^{1}_{t}(L^2)} \Bigr).  
\end{split}
\end{equation}
\noindent The smallness condition on $\| (S_{m}a-a) \|_{L^{\infty}_{t}(B^{\frac{3}{2}}_{2,1})}$ allows to write
\begin{equation*}
\begin{split}
\frac{b}{2}\, \|\nabla \Pi \|_{L^{1}_{t}(B^{\frac{1}{2}}_{2,1})} &\lesssim \int_{0}^{t} \| u(t') \|_{B^{\frac{1}{2}}_{2,1}}\, \| \nabla u(t') \|_{L^{\infty}}\,\, dt' +\,  \| (S_{m}a-a) \|_{L^{\infty}_{t}(B^{\frac{3}{2}}_{2,1})} \, \| u \|_{L^{1}_{t}(B^{\frac{5}{2}}_{2,1})}\\ &\quad + \| \nabla S_{m}a \|_{L^{\infty}_{t}(\dot{H}^1)}\, \Bigl( \|\nabla \Pi\|_{L^{1}_{t}(L^2)} + \| \Delta u\|_{L^{1}_{t}(L^2)} \Bigr).  
\end{split}
\end{equation*} 
\noindent Obviously, $\ds{ \| \nabla S_{m}a \|_{L^{\infty}_{t}(\dot{H}^1)} \leqslant C\, 2^{2m}\, \| a \|_{L^{\infty}_{t}(L^2)}}$. Therefore, 
\begin{equation}
\label{pressure en norme H1}
\begin{split}
\frac{b}{2}\, \|\nabla \Pi \|_{L^{1}_{t}(B^{\frac{1}{2}}_{2,1})} &\lesssim \int_{0}^{t} \| u(t') \|_{B^{\frac{1}{2}}_{2,1}}\, \| \nabla u(t') \|_{L^{\infty}}\,dt' +\,  \| (S_{m}a-a) \|_{L^{\infty}_{t}(B^{\frac{3}{2}}_{2,1})} \, \| u \|_{L^{1}_{t}(B^{\frac{5}{2}}_{2,1})}\\ &\quad +  2^{2m}\, \| a \|_{L^{\infty}_{t}(L^2)} \, \Bigl( \|\nabla \Pi\|_{L^{1}_{t}(L^2)} + \|  u\|_{L^{1}_{t}(\dot{H}^2)} \Bigr).  
\end{split}
\end{equation}
\noindent This ends up the estimate on the pressure term in $\ds{L^{1}_{t}(B^{\frac{1}{2}}_{2,1})}$-norm. It is left with estimate the pressure term in the  $L^{1}_{t}(L^{2})$-norm, in order to get rid of it in the above estimate, and thus, it is likely to applying with success Gronwall Lemma in the estimate of the velocity term. 

\vskip 0.3cm
\noindent \textit{Step 4: Estimate of $\|  \nabla \Pi \|_{L^{1}_{t}(L^2)}$}.\\
\noindent Once again, we take the divergence in the momentum equation, and the $\dot{H}^{-1}$-norm, so that we get
\begin{equation*}
\begin{split}
\| \dive\bigl( (1+S_{m}a)\nabla \Pi \bigr)\|_{\dot{H}^{-1}} &\leqslant \|\dive(u\cdotp \nabla u)\|_{\dot{H}^{-1}} + \| \Delta u \cdotp \nabla S_{m}a \|_{\dot{H}^{-1}}+ \|\dive \Bigl( (S_{m}a-a)(\nabla \Pi - \Delta u)  \Bigr)\|_{\dot{H}^{-1}}.
\end{split} 
\end{equation*}
\noindent We recall that the smallness condition implies that $\ds{(1+S_{m}a) \geqslant \frac{b}{2}}$ and thus 
$$ b\,  \| \nabla \Pi \|_{L^2} \leqslant\, C\, \| (1+S_{m}a)\nabla \Pi \|_{L^2} \leqslant\, C\, \|u\cdotp \nabla u\|_{L^{2}} + \| \Delta u \cdotp \nabla S_{m}a \|_{\dot{H}^{-1}}\, +\, \|  (S_{m}a-a)(\nabla \Pi - \Delta u) \|_{L^2}.$$ 
Thanks to the smallness condition and product law, we have
\begin{equation}
\begin{split}
 \frac{b}{2}  \| \nabla \Pi \|_{L^2} &\lesssim \|u\|_{L^3}\, \| \nabla u\|_{L^{6}} + \| \Delta u \cdotp \nabla S_{m}a \|_{\dot{H}^{-1}}\,   + \, \| (S_{m}a-a) \Delta u \|_{L^2}.
\end{split} 
\end{equation}
\noindent On the one hand, Gagliardo-Niremberg inequality (notice that average of $\nabla u$ is nul) yields
\begin{equation*}
\begin{split}
 \frac{b}{2}  \| \nabla \Pi \|_{L^2} &\lesssim \|u\|_{L^3}\, \| \nabla^2 u\|_{L^{2}} +  \| \Delta u \cdotp \nabla S_{m}a \|_{\dot{H}^{-1}}\,     + \, \| a\|_{L^{\infty}} \|  \Delta u \|_{L^2}.
\end{split} 
\end{equation*}
\noindent On the other hand, we prove easily thanks to the divergence free condition that
\begin{equation*}
 \| \Delta u \cdotp \nabla S_{m}a \|_{\dot{H}^{-1}}\, \leq C\ \|a\|_{L^{\infty}}\, \| \Delta u \|_{L^2}.
\end{equation*}
\noindent Despite the fact that average of $u$ is not nul, we have  $\ds{\|u\|_{L^3}\, \leqslant C(\rho_{0})\, \|u\|_{B^{\frac{1}{2}}_{2,1}}\, }$. Hence, one has
\begin{equation}
\label{pression en nome L2}
\begin{split}
 \frac{b}{2}  \| \nabla \Pi \|_{L^{1}_{t}(L^2)} &\lesssim \bigl( \|u\|_{L^{\infty}_{t}(B^{\frac{1}{2}}_{2,1})}\, + 2\|a\|_{L^{\infty}_{t}(L^{\infty})} \bigr) \|  u\|_{L^{1}_{t}(\dot{H}^{2})}.
\end{split} 
\end{equation}
\noindent Plugging (\ref{pression en nome L2}) in the estimate (\ref{pressure en norme H1}), we finally get an estimate of the pressure, in which the right-hand side is independent of the pressure: we got rid of the term $\ds{\| \nabla \Pi \|_{L^2}}$. Indeed, (\ref{pressure en norme H1}) becomes
\begin{equation}
\label{pressure en norme H1 bis}
\begin{split}
\frac{b}{2}\, \|\nabla \Pi \|_{L^{1}_{t}(B^{\frac{1}{2}}_{2,1})} &\lesssim \int_{0}^{t} \| u \|_{B^{\frac{1}{2}}_{2,1}}\, \| \nabla u \|_{L^{\infty}}\, dt' +\,  \| (S_{m}a-a) \|_{L^{\infty}_{t}(B^{\frac{3}{2}}_{2,1})} \, \| u \|_{L^{1}_{t}(B^{\frac{5}{2}}_{2,1})}\\ &\quad +  2^{2m}\, \| a \|_{L^{\infty}_{t}(L^2)}\, \|  u\|_{L^{1}_{t}(\dot{H}^{2})}  \, \Bigl( 1 +\, \|u\|_{L^{\infty}_{t}(B^{\frac{1}{2}}_{2,1})}\, + \|a\|_{L^{\infty}_{t}(L^{\infty})}  \Bigr).  
\end{split}
\end{equation} 

\noindent Plugging (\ref{pression en nome L2}) in the estimate  (\ref{estimate3}), we also get
\begin{equation}
\label{estimate3 nouvelle}
\begin{split}
\| u \|_{L^{\infty}_{t}(B^{\frac{1}{2}}_{2,1})}\, +\,C\,b\, \|  u \|_{L^{1}_{t}(B^{\frac{5}{2}}_{2,1})}\, &\leqslant\,\, \| u_{0} \|_{B^{\frac{1}{2}}_{2,1}}\, + \| a - S_{m}a \|_{L^{\infty}_{t}(B^{\frac{3}{2}}_{2,1})}\, \Bigl( \|  \nabla \Pi \|_{L^{1}_{t}(B^{\frac{1}{2}}_{2,1})} + \|  u \|_{L^{1}_{t}(B^{\frac{5}{2}}_{2,1})}    \Bigr)\\
&\qquad + \int_{0}^{t} \| \nabla u(t') \|_{L^{\infty}}\, \| u(t') \|_{B^{\frac{1}{2}}_{2,1}}\, dt'\, + \, 2^{m}\, \| a\|_{L^{\infty}_{t}(L^{\infty})} \, \| u \|_{L^{1}_{t}(B^{\frac{3}{2}}_{2,1})}\\
&\qquad + 2^{2m+1}\, \| a\|_{L^{\infty}_{t}(L^{2})} \, \| u \|_{L^{1}_{t}(\dot{H}^{2})} \Bigl( 1  + \, \|u\|_{L^{\infty}_{t}(B^{\frac{1}{2}}_{2,1})}\, + 2\|a\|_{L^{\infty}_{t}(L^{\infty})}  \Bigr),
\end{split}
\end{equation}

\noindent Suuming (\ref{estimate3 nouvelle}) with (\ref{pressure en norme H1 bis}) and using obvious estimates on the transport equation below
$$ \|a\|_{L^{\infty}_{t}(L^{\infty})} \leqslant \|a_{0}\|_{L^{\infty}} \quad \hbox{and} \quad \|a\|_{L^{\infty}_{t}(L^{2})} \leqslant \|a_{0}\|_{L^{2}},$$
\noindent leads to

\begin{equation*}
\begin{split}
\| u \|_{L^{\infty}_{t}(B^{\frac{1}{2}}_{2,1})}\, +&\,C\,b\, \|  u \|_{L^{1}_{t}(B^{\frac{5}{2}}_{2,1})}\, + \frac{b}{2}\, \|\nabla \Pi \|_{L^{1}_{t}(B^{\frac{1}{2}}_{2,1})} \lesssim\,\, \| u_{0} \|_{B^{\frac{1}{2}}_{2,1}}\\ &+ \| a - S_{m}a \|_{L^{\infty}_{t}(B^{\frac{3}{2}}_{2,1})}\, \bigl( \|  \nabla \Pi \|_{L^{1}_{t}(B^{\frac{1}{2}}_{2,1})} + \|  u \|_{L^{1}_{t}(B^{\frac{5}{2}}_{2,1})}    \bigr)\\
& + \int_{0}^{t} \| \nabla u(t') \|_{L^{\infty}}\, \| u (t')\|_{B^{\frac{1}{2}}_{2,1}}\, dt'\, + \, 2^{m}\, \| a_{0}\|_{L^{\infty}} \, \| u \|_{L^{1}_{t}(B^{\frac{3}{2}}_{2,1})}\\
&+ 2^{2m+1}\, \| a_{0}\|_{L^{2}} \, \| u \|_{L^{1}_{t}(\dot{H}^{2})} \, \bigl( 1  + \, \|u\|_{L^{\infty}_{t}(B^{\frac{1}{2}}_{2,1})}\, + \|a_{0}\|_{L^{\infty}}  \bigr)\\
&+ \| u \|_{L^{1}_{t}(\dot{H}^{2})} \bigl( \, \|u\|_{L^{\infty}_{t}(B^{\frac{1}{2}}_{2,1})}\, + 2\, \|a_{0}\|_{L^{\infty}}  \bigr).
\end{split}
\end{equation*} 
\noindent Once again, the smallness condition simplifies the above estimate 
\begin{equation}
\label{estimate4}
\begin{split}
\| u \|_{L^{\infty}_{t}(B^{\frac{1}{2}}_{2,1})}\, +\,C\,b\, \|  u \|_{L^{1}_{t}(B^{\frac{5}{2}}_{2,1})}\, +& \frac{b}{2}\, \|\nabla \Pi \|_{L^{1}_{t}(B^{\frac{1}{2}}_{2,1})} \lesssim\,\, \| u_{0} \|_{B^{\frac{1}{2}}_{2,1}}\, +\, 2\,\int_{0}^{t} \| u(t') \|_{B^{\frac{5}{2}}_{2,1}}\, \| u(t') \|_{B^{\frac{1}{2}}_{2,1}}\, dt'\\ &+\, \bigl(1+ \, 2^{2m+1}\, \| a_{0} \|_{L^2}\bigr) \,  \|  u\|_{L^{1}_{t}(\dot{H}^{2})}  \, \bigl( 1 +\, \|u\|_{L^{\infty}_{t}(B^{\frac{1}{2}}_{2,1})}\, + \|a_{0}\|_{L^{\infty}}  \bigr)\\
&+ 2^{m}\, \| a_{0}\|_{L^{\infty}} \, \| u \|_{L^{1}_{t}(B^{\frac{3}{2}}_{2,1})}.  
\end{split}
\end{equation}
\noindent Let us recall somme interpolation properties. The following inequalities hold on the torus:
$$ \| u\|_{\dot{H}^2} \leqslant C\, \| u \|^{\frac{1}{4}}_{B^{\frac{1}{2}}_{2,1}}\, \| u \|^{\frac{3}{4}}_{B^{\frac{5}{2}}_{2,1}} \quad \hbox{and} \quad \| u\|_{B^{\frac{3}{2}}_{2,1}} \leqslant C\, \| u \|^{\frac{1}{2}}_{B^{\frac{1}{2}}_{2,1}}\, \| u \|^{\frac{1}{2}}_{B^{\frac{5}{2}}_{2,1}}.$$
\noindent They are due the product laws in Besov spaces (cf Lemma \ref{lemma product law}). For instance, the first one stems from
\begin{equation*}
\| u\|_{\dot{H}^2} = \| \nabla u \|_{\dot{H}^1} \leqslant \| \nabla u \|_{H^1}  \leqslant \| \nabla u \|_{B^{1}_{2,1}} \leqslant C\, \| \nabla u \|^{\frac{1}{4}}_{B^{-\frac{1}{2}}_{2,1}} \, \| \nabla u \|^{\frac{3}{4}}_{B^{\frac{3}{2}}_{2,1}}.
\end{equation*}
\noindent Obviously, by integration in time and thanks to Hölder's inequality, we have
$$ \| u\|_{L^{1}_{t}(\dot{H}^2)} \leqslant C\, \| u \|^{\frac{1}{4}}_{L^{1}_{t}(B^{\frac{1}{2}}_{2,1})}\, \| u \|^{\frac{3}{4}}_{L^{1}_{t}(B^{\frac{5}{2}}_{2,1})} \quad \hbox{and} \quad \| u\|_{L^{1}_{t}(B^{\frac{3}{2}}_{2,1})} \leqslant C\, \| u \|^{\frac{1}{2}}_{L^{1}_{t}(B^{\frac{1}{2}}_{2,1})}\, \| u \|^{\frac{1}{2}}_{L^{1}_{t}(B^{\frac{5}{2}}_{2,1})}.$$
\noindent By vertue of Young's inequalities $$ xy \leqslant \frac{x^{4}}{4} + \frac{3\,y^{\frac{4}{3}}}{4} \quad \hbox{ and} \quad  xy \leqslant \frac{x^{2}}{2} + \frac{y^{2}}{2},$$
\noindent Estimate (\ref{estimate4}) becomes 
\begin{equation*}
\begin{split}
\| u \|_{L^{\infty}_{t}(B^{\frac{1}{2}}_{2,1})}\, &+\,C\,b\, \|  u \|_{L^{1}_{t}(B^{\frac{5}{2}}_{2,1})}\, + \frac{b}{2}\, \|\nabla \Pi \|_{L^{1}_{t}(B^{\frac{1}{2}}_{2,1})} \lesssim\,\, \| u_{0} \|_{B^{\frac{1}{2}}_{2,1}}\, +\, 2\,\int_{0}^{t} \| \nabla u(t') \|_{L^{\infty}}\, \| u(t') \|_{B^{\frac{1}{2}}_{2,1}}\, dt'\\    &\qquad +\, \big(1+ \,  2^{8m}\, \| a_{0} \|^{4}_{L^2} \bigr)\, \|  u\|_{L^{1}_{t}(B^{\frac{1}{2}}_{2,1})}  \, \bigl( 1 +\, \|u\|^{4}_{L^{\infty}_{t}(B^{\frac{1}{2}}_{2,1})}\, + \|a_{0}\|^{4}_{L^{\infty}}  \bigr) \, +\, \frac{b}{4}\, \| u \|_{L^{1}_{t}(B^{\frac{5}{2}}_{2,1})}\\
&\qquad+ \, 2^{2m}\, \| a_{0}\|^{2}_{L^{\infty}} \, \| u \|_{L^{1}_{t}(B^{\frac{1}{2}}_{2,1})}\, +\frac{b}{4} \| u \|_{L^{1}_{t}(B^{\frac{5}{2}}_{2,1})}. 
\end{split}
\end{equation*}
\noindent which can be simplified by
\begin{equation*}
\begin{split}
\| u \|_{L^{\infty}_{t}(B^{\frac{1}{2}}_{2,1})}\, +\,C\,\frac{b}{2}\,  \|  u \|_{L^{1}_{t}(B^{\frac{5}{2}}_{2,1})}\, +& \frac{b}{2}\, \|\nabla \Pi \|_{L^{1}_{t}(B^{\frac{1}{2}}_{2,1})} \lesssim\,\, \| u_{0} \|_{B^{\frac{1}{2}}_{2,1}}\, +\, 2\,\int_{0}^{t} \| \nabla u(t') \|_{L^{\infty}}\, \| u (t')\|_{B^{\frac{1}{2}}_{2,1}}\, dt'\\     &\qquad + \, \big(1+ \,  2^{8m}\, \| a_{0} \|^{4}_{L^2} \bigr)\, \|  u\|_{L^{1}_{t}(B^{\frac{1}{2}}_{2,1})}  \, \bigl( 1 +\, \|u\|^{4}_{L^{\infty}_{t}(B^{\frac{1}{2}}_{2,1})}\, + \|a_{0}\|^{4}_{L^{\infty}}  \bigr)\\ &\qquad+ \, 2^{2m}\, \| a_{0}\|^{2}_{L^{\infty}} \, \| u \|_{L^{1}_{t}(B^{\frac{1}{2}}_{2,1})}.
\end{split}
\end{equation*}
\noindent This concludes the proof of  Lemma \ref{lemma general avec linfini}. \\
\noindent \textit{Continuation of the proof of existence part of Theorem \ref{theorem LWP}}. 
This stems from the obvious fact : $\ds{B^{\frac{3}{2}}_{2,1}  \hookrightarrow  L^{\infty}}$ and thus 
$$   \|  \nabla u\|_{L^{\infty}}  \leqslant \|  \nabla u\|_{B^{\frac{3}{2}}_{2,1}}.$$
\noindent Therefore, we get 
\begin{equation*}
\begin{split}
\| u \|_{L^{\infty}_{t}(B^{\frac{1}{2}}_{2,1})}\, +\,C\,b\, \|  u \|_{L^{1}_{t}(B^{\frac{5}{2}}_{2,1})}\, +& \frac{b}{2}\, \|\nabla \Pi \|_{L^{1}_{t}(B^{\frac{1}{2}}_{2,1})} \lesssim\,\, \| u_{0} \|_{B^{\frac{1}{2}}_{2,1}}\, +\, 2\,\int_{0}^{t} \| u(t') \|_{B^{\frac{5}{2}}_{2,1}}\, \| ut(t') \|_{B^{\frac{1}{2}}_{2,1}}\, dt'\\          &\qquad + \, \big(1+ \,  2^{8m}\, \| a_{0} \|^{4}_{L^2} \bigr)\, \|  u\|_{L^{1}_{t}(B^{\frac{1}{2}}_{2,1})}  \, \bigl( 1 +\, \|u\|^{4}_{L^{\infty}_{t}(B^{\frac{1}{2}}_{2,1})}\, + \|a_{0}\|^{4}_{L^{\infty}}  \bigr)\\ &\qquad+ \, 2^{2m}\, \| a_{0}\|^{2}_{L^{\infty}} \, \| u \|_{L^{1}_{t}(B^{\frac{1}{2}}_{2,1})}.
\end{split}
\end{equation*}

\noindent As a result, we get 
\begin{equation}
\begin{split}
\| u \|_{L^{\infty}_{t}(B^{\frac{1}{2}}_{2,1})}\, +&\,C\, \frac{b}{2}\, \|  u \|_{L^{1}_{t}(B^{\frac{5}{2}}_{2,1})}\, + \frac{b}{2}\, \|\nabla \Pi \|_{L^{1}_{t}(B^{\frac{1}{2}}_{2,1})} \lesssim\,\, \| u_{0} \|_{B^{\frac{1}{2}}_{2,1}}\, +\, 2\,\int_{0}^{t} \| u \|_{B^{\frac{5}{2}}_{2,1}}\, \| u \|_{B^{\frac{1}{2}}_{2,1}}\, dt'\\ &\qquad+ \int_{0}^{t} \| u(t') \|_{B^{\frac{1}{2}}_{2,1}} \Bigl(  2^{2m}\, \| a_{0}\|^{2}_{L^{\infty}} \, +\, \bigl( 1+ \, 2^{8m}\, \|a_{0}\|^{4}_{L^{2}}\bigr)\, \bigl( 1 +\, \|u\|^{4}_{L^{\infty}_{t}(B^{\frac{1}{2}}_{2,1})}\, + \|a_{0}\|^{4}_{L^{\infty}}  \bigr) \Bigr)\,dt'.
\end{split}
\end{equation}
\noindent Let $\varepsilon_{0} >0$. Let us introduce the time $T_{0}$ such that 
$$ T_{0} \, \eqdefa \, \sup\bigl\{0\leqslant  t \leqslant T^{*} \,\, | \,\,  \| u(t) \|_{B^{\frac{1}{2}}_{2,1}}   \leqslant \varepsilon_{0} \bigr\}.$$
\noindent Hence, for any $t \leqslant T_{0}$, we have 
\begin{equation*}
\begin{split}
\| u \|_{L^{\infty}_{t}(B^{\frac{1}{2}}_{2,1})}\, +&\,C\,\frac{b}{2}\, \|  u \|_{L^{1}_{t}(B^{\frac{5}{2}}_{2,1})}\, + \frac{b}{2}\, \|\nabla \Pi \|_{L^{1}_{t}(B^{\frac{1}{2}}_{2,1})} \lesssim\,\, \| u_{0} \|_{B^{\frac{1}{2}}_{2,1}}\, +\, 2\, \varepsilon_{0}\, \int_{0}^{t} \| u(t') \|_{B^{\frac{5}{2}}_{2,1}}\,dt'\\ &\qquad+ \int_{0}^{t} \| u(t') \|_{B^{\frac{1}{2}}_{2,1}}\, \Bigl(  2^{2m}\, \| a_{0}\|^{2}_{L^{\infty}} \, +\, \bigl( 1+ \, 2^{8m}\, \|a_{0}\|^{4}_{L^{2}}\bigr)\, \bigl( 1 +\, \varepsilon_{0}^{4}\, + \|a_{0}\|^{4}_{L^{\infty}}  \bigr) \Bigr)\,dt'.
\end{split}
\end{equation*}
\noindent Choosing $\varepsilon_{0}$ small enough, namely $\ds{\varepsilon_{0} \leqslant \frac{C\, b}{4}}$, Gronwall lemma implies that for any $t \leqslant T_{0}$, 
\begin{equation}
\begin{split}
\| u \|_{L^{\infty}_{t}(B^{\frac{1}{2}}_{2,1})}\, +\,\frac{C\, b}{4}\, \|  u \|_{L^{1}_{t}(B^{\frac{5}{2}}_{2,1})}\, +& \frac{b}{2}\, \|\nabla \Pi \|_{L^{1}_{t}(B^{\frac{1}{2}}_{2,1})} \lesssim\,\, \| u_{0} \|_{B^{\frac{1}{2}}_{2,1}}\\ &\times \exp{ \big((T_{0}\, \bigl(  2^{2m}\, \| a_{0}\|^{2}_{L^{\infty}} \, +\, \bigl( 1+ \, 2^{8m}\, \|a_{0}\|^{4}_{L^{2}}\bigr)\, \bigl( 1 +\, {\bigl(\frac{b}{4}\bigr)}^{4}\, + \|a_{0}\|^{4}_{L^{\infty}}  \bigr) \bigr)}.
\end{split}
\end{equation}
\noindent As a result, we get the a priori estima on the velocity 
\begin{equation}
\begin{split}
\hbox{For any} \quad t \leqslant T_{0},\quad \| u \|_{L^{\infty}_{t}(B^{\frac{1}{2}}_{2,1})}\, +\,\|  u \|_{L^{1}_{t}(B^{\frac{5}{2}}_{2,1})}\, +& \frac{b}{2}\, \|\nabla \Pi \|_{L^{1}_{t}(B^{\frac{1}{2}}_{2,1})} \leqslant\,C\, \| u_{0} \|_{B^{\frac{1}{2}}_{2,1}}.\\
\end{split}
\end{equation}
\noindent This concludes the proof of (\ref{velocity}) : until the (small) time $T_{0}$, the solution is controlled by initial data, up to a multiplicative constant.
\noindent This ends up the proof of the local-existence part of Theorem \ref{theorem LWP}. 

\subsection{Uniqueness part}
\noindent The uniqueness part has been already done in \cite{AGZ}. We refer the reader to it for more details. Let us recall some details. Let $(a_{1},u_{1},\nabla \Pi_{1})$ and $(a_{2},u_{2},\nabla \Pi_{2})$ be two solutions of the system (\ref{NSIH avec a}), satisfying the smallness hypothesis $\| a - S_{m}a \|_{B^{\frac{3}{2}}_{2,1}} \leqslant c$ and such that 
\begin{equation}
(a_{i},u_{i},\nabla \Pi_{i}) \in \,\, \mathcal{C}([0,T],B^{\frac{3}{2}}_{2,1})\, \times\, \mathcal{C}([0,T],B^{\frac{1}{2}}_{2,1})\,\cap\, L^{1}([0,T],B^{\frac{5}{2}}_{2,1}) \times\, L^{1}([0,T],B^{\frac{1}{2}}_{2,1})\cdotp
\end{equation}
\noindent We define as one expects 
$$ (\delta a, \delta u, \nabla \delta \Pi) \eqdefa (a_{2}- a_{1}, u_{2}-u_{1}, \nabla \Pi_{2} - \nabla \Pi_{1} ),$$
\noindent so that  $(\delta a, \delta u, \nabla \delta \Pi) $ solves the following system
\begin{equation}
\label{NSIH unicite} \left \lbrace \begin {array}{ccc}  \partial_{t} \delta a\, + u_{2}\cdotp\nabla{\delta a} &=& -\delta u \cdot \nabla a_{1}\\ \partial_{t} \delta u + u_{2}\cdot\nabla{\delta u} -(1+a_{2})\, (- \nabla{ \delta \Pi} -+\Delta{\delta u} )&=&- \delta u\cdot \nabla u_{1} + \delta a (\Delta u_{1} - \nabla \Pi_{1})\\
 \dive \delta u&=&0\\
 (\delta a,\delta u)_{|t=0}&=&(0,0).\\ \end{array}
\right.
\end{equation}
\noindent We prove that such solution of this system satisifies
\begin{equation}
(\delta a, \delta u, \nabla \delta \Pi) \in ,\, \mathcal{C}([0,T],B^{\frac{3}{2}}_{2,1})\, \times\, \mathcal{C}([0,T],B^{-\frac{1}{2}}_{2,1})\,\cap\, L^{1}([0,T],B^{\frac{3}{2}}_{2,1}) \times\, L^{1}([0,T],B^{-\frac{1}{2}}_{2,1})\cdotp
\end{equation}
\begin{remark}
\noindent Notice that, owing to the presence of a transport equation, we loose one derivative  in the estimate involving $\delta a$. 
\end{remark}

\section{Proof of the global wellposedness part of the main theorem}
\noindent This section is devoted to the proof of Theorem \ref{theorem1GWP}, which provides the global property of the main Theorem \ref{main theorem}. 

\begin{equation} \left \lbrace \begin {array}{ccc} \label{system NSI 3}  \partial_{t} \rho\, + u\cdotp\nabla{\rho} &=& 0\\ \rho(\partial_{t} u + u\cdot\nabla{u})  - \Delta{u} + \nabla{\Pi}&=&0\\
 \dive u&=&0\\
 (\rho,u)_{|t=0}&=&(\rho_{0},u_{0}).\\ \end{array}
\right.
\end{equation}
\noindent In a sake of simplicity, we skip the regularisation process (Friedrich methods) and we only present the a priori estimates for smooth enough solution $(\rho,u)$, which provide the existence part of Theorem \ref{theorem1GWP}. Concerning the uniqueness part, we refer the reader to the paper of M. Paicu, P. Zhang and Z. Zhang (see \cite{PZZ}). We underline that Lagragian coordinates are necessary to prove the uniqueness, owing to the very low regularity hypothesis on the density( which is only supposed to be bounded from above and from below). Let us proceed firstly to an $L^2$-energy estimate, which leads to the result on $B_{0}$. Then we will get estimate on $B_{1}$, thanks to an $H^1$-energy estimate. 
\vskip 0.2cm
\noindent $\bullet$\,\, Proof of (\ref{L2 energy estimate}). Taking the $L^2$ inner product of momentum equation with $u$ in the system (\ref{system NSI 3}), we get : 
\begin{equation*}  
\bigl({\, \rho(\partial_{t} u + u\cdot\nabla{u}) \,\vert \, u\,}\bigr)_{L^2} - \bigl({\,  \Delta{u} \,\vert \, u\,}\bigr)_{L^2} + 0 = 0.
\end{equation*}
\noindent We check that $\dis{\bigl({\, \rho(\partial_{t} u + u\cdot\nabla{u}) \, \vert \, u\,}\bigr)_{L^2} = \frac{1}{2} \dfrac{d}{dt}\| \sqrt{\rho} u \|^2_{L^2}}$. \\ 
\noindent This stems from the computations below
\begin{equation*}
\begin{split}
\bigl({\, \rho(\partial_{t} u + u\cdot\nabla{u}) \,\vert \, u\,}\bigr)_{L^2} &= \frac{1}{2}\, \int_{\T^3} \rho \, \partial_{t} |u|^2 \, dx\, + \frac{1}{2}\int_{\T^3} \rho \,\, u\cdotp\nabla {|u|^2}\, dx\\
&= \frac{1}{2} \dfrac{d}{dt} \int_{\T^3} \rho \,|u|^2 -\frac{1}{2}\, \int_{\T^3} \partial_{t}\rho\, |u|^2 \, dx + \frac{1}{2}\int_{\T^3} \rho \,\, u\cdotp\nabla {|u|^2}\, dx.\\
\end{split}
\end{equation*}
\noindent However, $\ds{ \int_{\T^3} \rho \, u\cdotp\nabla {|u|^2} = - \int_{\T^3}  \, (u\cdotp\nabla{\rho}) |u|^2}$. Therefore, the transport equation yields
\begin{equation*}
\begin{split}
\bigl({\, \rho(\partial_{t} u + u\cdot\nabla{u}), u\,}\bigr)_{L^2}
&= \frac{1}{2} \dfrac{d}{dt} \int_{\T^3} \rho \,|u|^2 -\frac{1}{2}\, \int_{\T^3} (\partial_{t}\rho + u\cdotp\nabla{\rho})  \, |u|^2 \\
&= \frac{1}{2} \dfrac{d}{dt} \int_{\T^3} \rho \,|u|^2.
\end{split}
\end{equation*}
Finally, an integration in time provides the desired estimate
\begin{equation}
\frac{1}{2} \| \sqrt{\rho} u(t)\|^{2}_{L^2} +  \int_{0}^{t} \| \nabla{u}(t')\|^{2}_{L^2} \,dt'= \frac{1}{2} \| \sqrt{\rho_{0}} u_{0}\|^{2}_{L^2}.
\end{equation}
\noindent This concludes the proof of (\ref{L2 energy estimate}). Now let us proceed to the proof of (\ref{H1 energy estimate}).

\vskip 0.2cm
\noindent $\bullet$\,\, Proof of (\ref{H1 energy estimate}). The idea is the same as the previous one : we take the $L^2$ inner product of momentum equation with $\partial_{t} u$ in the system (\ref{system NSI 3}), we get :
\begin{equation*}
\begin{split}  
\bigl({\, \sqrt{\rho}\,\partial_{t}u \,\vert \,\sqrt{\rho}\partial_{t} u\,}\bigr)_{L^2} + \bigl({\, \sqrt{\rho} u\cdot\nabla{u} \,\vert \,\sqrt{\rho}\partial_{t} u\,}\bigr)_{L^2}\, + \frac{1}{2}\,\dfrac{d}{dt}\| \nabla{u}(t)\|^{2}_{L^2}  &= 0,
\end{split}
\end{equation*}
\noindent which leads to 
\begin{equation}
\begin{split} 
\frac{1}{2}\,\dfrac{d}{dt}\| \nabla{u}(t)\|^{2}_{L^2}\, +\, \|  \sqrt{\rho}\,\partial_{t}u(t)  \|^{2}_{L^2} &\leqslant \| \sqrt{\rho} u\cdot\nabla{u} (t)\|_{L^2}\, \| \sqrt{\rho}\partial_{t} u (t)\|_{L^2}\\
&\leqslant \| \sqrt{\rho} u(t) \|_{L^6}\, \| \nabla{u}(t) \|_{L^3}\, \| \sqrt{\rho}\partial_{t} u (t)\|_{L^2}\\
\end{split}
\end{equation}

\noindent Applying Proposition \ref{application1 lemme average} on the term $\| u(t) \|_{L^6}\,$  and Proposition \ref{GN} on the term $\| \nabla{u}(t) \|_{L^3}$ gives rise to
\begin{equation}
\label{inegalit\'e cl\'e}
\begin{split}  
\frac{1}{2}\,\dfrac{d}{dt}\| \nabla{u}(t)\|^{2}_{L^2}\, +\, \|  \sqrt{\rho}\,\partial_{t}u (t) \|^{2}_{L^2} 
 &\leqslant C(\rho_{0})\,  \| \nabla{ u}(t) \|_{L^2}\, \| \nabla{u}(t) \|^{\frac{1}{2}}_{L^2}\, \| \nabla^2{u}(t) \|^{\frac{1}{2}}_{L^2}\,\,  \| \sqrt{\rho}\partial_{t} u (t)\|_{L^2}\\
& \leqslant C(\rho_{0})\,  \| \nabla{u}(t) \|^{\frac{3}{2}}_{L^2}\, \| \nabla^2{u}(t) \|^{\frac{1}{2}}_{L^2}\,\,  \| \sqrt{\rho}\partial_{t} u(t) \|_{L^2}.\\
\end{split}
\end{equation}

\noindent Then, Young inequality yields 
\begin{equation}
\label{inegalit\'e cl\'e}
\begin{split}  
\frac{1}{2}\,\dfrac{d}{dt}\| \nabla{u}(t)\|^{2}_{L^2}\, +\, \frac{1}{2}\,  \|  \sqrt{\rho}\,\partial_{t}u(t)  \|^{2}_{L^2}  
& \leq  \frac{1}{2}\, C(\rho_{0})\,  \| \nabla{u}(t) \|^{3}_{L^2}\, \| \nabla^2{u}(t) \|_{L^2}.
\end{split}
\end{equation}

\noindent We have to estimate the term $ \| \nabla^2{u} \|_{L^2}$. Applying the   $L^2$-norm in the momentum equation, we get
\begin{equation*}
\begin{split}
 \|\nabla^2{u}(t)\|_{L^2} + \|\nabla{\Pi}(t)\|_{L^2} &\leqslant \|\rho(t)\|^{\frac{1}{2}}_{L^\infty}\Bigl( \| \sqrt{\rho}\partial_{t} u(t)\|_{L^2} + \| \sqrt{\rho} u(t)\|_{L^6}\| \nabla{u}\|_{L^3}  \Bigr).
\end{split}
\end{equation*}
Once again, by vertue of Proposition \ref{application1 lemme average} and  Gagliardo-Niremberg inequality, one has
\begin{equation*}
\begin{split}
 \|\nabla^2{u}\|_{L^2} + \|\nabla{\Pi}\|_{L^2} &\leqslant  C(\rho_{0})\, \Bigl( \| \sqrt{\rho}\partial_{t} u(t)\|_{L^2} + \,   \| \nabla{u}(t) \|_{L^2}\, \| \nabla{u}(t)\|^{\frac{1}{2}}_{L^2}\,\, \| \nabla^2{u}(t)\|^{\frac{1}{2}}_{L^2}  \Bigr).\\
\end{split}
\end{equation*}
Young inequality implies  
\begin{equation}
\label{inequation 3}
\frac{1}{2} \|\nabla^2{u}(t)\|_{L^2} + \|\nabla{\Pi}\|_{L^2} \leqslant C(\rho_{0})\, \| \sqrt{\rho}\partial_{t} u(t)\|_{L^2} +  \frac{1}{2}\, \| \nabla{u}(t)\|^3_{L^2}.
\end{equation}
\noindent Plugging Inequality (\ref{inequation 3}) in (\ref{inegalit\'e cl\'e}) and applying Young inequality gives
\begin{equation}
\label{inequation totale}
\begin{split}
\frac{1}{2}\,\dfrac{d}{dt}\| \nabla{u}(t)\|^{2}_{L^2}\, +\, \frac{1}{2}\|  \sqrt{\rho}\,\partial_{t}u(t)  \|^{2}_{L^2} &\leqslant
 \, C(\rho_{0})\,  \| \nabla{u}(t) \|^{6}_{L^2}\, + \frac{1}{4}\|  \sqrt{\rho}\,\partial_{t}u (t) \|^{2}_{L^2}.
\end{split}
\end{equation}
As a result, we have : 
\begin{equation}
\label{inequation totale bis}
\begin{split}
\frac{1}{2}\,\dfrac{d}{dt}\| \nabla{u}(t)\|^{2}_{L^2}\, +\, \frac{1}{4}\|  \sqrt{\rho}\,\partial_{t}u(t)  \|^{2}_{L^2} &\leqslant\, C(\rho_{0})\ \| \nabla{u}(t) \|^{6}_{L^2}.
\end{split}
\end{equation}
We sum (\ref{inequation totale bis}) and (\ref{inequation 3}) and we get : 
\begin{equation*}
\begin{split}
\frac{1}{2}\,\dfrac{d}{dt}\| \nabla{u}(t)\|^{2}_{L^2}\, +\, \frac{1}{4}\|  \sqrt{\rho}\,\partial_{t}u(t)  \|^{2}_{L^2} +\,  \|\nabla^2{u}(t)\|^{2}_{L^2} + \|\nabla{\Pi}(t)\|^{2}_{L^2}\, &\leqslant C(\rho_{0})\, \bigl(\| \nabla{u}(t) \|^{6}_{L^2} \\&+ \frac{1}{8} \, \| \sqrt{\rho}\partial_{t} u(t)\|^2_{L^2} +  \, \| \nabla{u}(t)\|^6_{L^2}\bigr).
\end{split}
\end{equation*}
Finally, we have by integration in time 
\begin{equation}
\label{inequation totale integr\'ee}
\begin{split}
\frac{1}{2}\, \| \nabla{u}(t)\|^{2}_{L^2}\, +\, \int_{0}^{t} \Bigl(  \frac{1}{8} \, \| \sqrt{\rho}\,\partial_{t}u(t')  \|^{2}_{L^2} +\,\|\nabla^2{u}(t')\|^{2}_{L^2} + \|\nabla{\Pi}(t')\|^{2}_{L^2}\Bigr) \,dt'\, &\leqslant \,  \,\frac{1}{2}\, \| \nabla{u_{0}}\|^{2}_{L^2}\\ &+\,  C(\rho_{0}) \,\int_{0}^{t} \| \nabla{u}(t')\|^6_{L^2}\, dt'.
\end{split}
\end{equation}
\noindent Let us focus for a while on the term $\ds{\int_{0}^{t} \| \nabla{u}(t')\|^6_{L^2} \, dt'}$.
\noindent It seems clear that 
\begin{equation*}
\int_{0}^{t} \| \nabla{u}(t')\|^6_{L^2} dt' \leqslant \, \| \nabla{u}\|^4_{L^{\infty}_{t}(L^2)} \, \int_{0}^{t} \| \nabla{u}(t')\|^2_{L^2}dt',
\end{equation*}
\noindent which leads to, by vertue of (\ref{H1 energy estimate}) and définition of $B_{1}$
\begin{equation*}
\int_{0}^{t} \| \nabla{u}(t')\|^6_{L^2} dt' \leqslant \,  \, \| u_{0}\|^{2}_{L^2}\, B^{2}_{1}(t).
\end{equation*}
\noindent Finally, we get
\begin{equation*}
 B_{1}(t) \leqslant \,\frac{1}{2}\, \| \nabla{u_{0}}\|^{2}_{L^2}\, +\,  C(\rho_{0}) \,\| u_{0}\|^{2}_{L^2}\, B^{2}_{1}(t).
\end{equation*}
\noindent As long as the smallness condition on $u_{0}$ is satisfied, we obtain Estimate (\ref{H1 energy estimate}), which conclude the proof of this estimate.
\vskip 0.2cm

\noindent $\bullet$ Proof of (\ref{H2 energy estimate}). Firstly, we derive the momentum equations, with respect to the time $t$. Then, we take the $L^2$ inner product with $\partial_{t} u$. 
\vsd
\noindent The derivated momentum equation is given by the following formula : 
\begin{equation*}
\begin{split}
\bigl({\, \rho \,\partial_{tt} u\, \vert\,  \partial_{t} u\,}\bigr)_{L^2} - \bigl({\,  \Delta{\partial_{t} u}\, \vert\,  \partial_{t} u\,}\bigr)_{L^2} &= -\bigl({\, \partial_{t} \rho\,\, (\partial_{t} u + u\cdotp \nabla{u}) \, \vert\,  \partial_{t} u\,}\bigr)_{L^2} - \bigl({\, \rho \, \partial_{t} u \cdot\nabla{u} \, \vert\, \partial_{t} u\,}\bigr)_{L^2}\\ &- \bigl({\, \rho \, u \cdot\nabla{\partial_{t} u} \, \vert\,  \partial_{t} u\,}\bigr)_{L^2}.
\end{split} 
\end{equation*}
By hypothesis on the density, the left-hand side can be bounded from below by :
\begin{equation*}
\begin{split}
\frac{m}{2}\, \dfrac{d}{dt} \bigl(\, \| \partial_{t} u \|^{2}_{L^2} \bigr)  + \, \, \|\nabla{\partial_{t} u}\|^{2}_{L^2}  &\leq \frac{m}{2}\,\| \partial_{t} u \|^{2}_{L^2}\, - \bigl({\, \rho \, \, \partial_{t} u \cdot\nabla{u}\,\vert\, \partial_{t} u\,}\bigr)_{L^2}\,  - \bigl({\, \rho \,\, u \cdotp\nabla{\partial_{t} u}\, \vert\, \partial_{t} u\,}\bigr)_{L^2}\\ &-\bigl({\,\, \partial_{t} \rho\,\, \partial_{t} u \, \vert\, \partial_{t} u\,}\bigr)_{L^2}\,  -\bigl({\,\, \partial_{t} \rho\,\, u\cdotp \nabla{u}) \,\vert\, \partial_{t} u\,}\bigr)_{L^2}.
\end{split}   
\end{equation*}

\noindent Let us point out that $\ds{\bigl({\, \rho \,\, u \cdotp\nabla{\partial_{t} u}\, \vert\,  \partial_{t} u\,}\bigr)_{L^2}}$ is in fact nul, by vertue of the divergence free condition.

\noindent Taking the modulus, applying triangular inequality and finally, using the mass equation on the density: 
\begin{equation}
\label{estimatederiv\'ee}
\begin{split}
\frac{m}{2}\, \dfrac{d}{dt} \bigl(\, \| \partial_{t} u \|^{2}_{L^2} \bigr)  + \, \, \|\nabla{\partial_{t} u}\|^{2}_{L^2}  &\leq \frac{m}{2}\,\| \partial_{t} u \|^{2}_{L^2}\, +\, \int_{\T^3} \Bigl| \, \rho \, (\partial_{t} u \cdot\nabla{u})\, \partial_{t} u \Bigr|\, dx  \\ &+ \Bigl| \bigl({\, \,\dive(\rho\,u)\,  \, \vert\, (\partial_{t} u)^2\,}\bigr)_{L^2}\Bigr|  + \Bigl|\bigl({\,\, \dive(\rho\,u)\, u\cdotp \nabla{u}) \, \vert\,  \partial_{t} u\,}\bigr)_{L^2}\Bigr|   \\
&\leqslant \sum_{k=1}^6 \, I_{k}(t),\\
\end{split}
\end{equation}
\noindent with
\begin{equation}
\begin{split}
I_{1}(t) &= \frac{m}{2}\,\| \partial_{t} u \|^{2}_{L^2}\, dx,\\
I_{2}(t)  &=\int_{\T^3} \Bigl| \rho \,\, (\partial_{t} u \cdot\nabla{u})\, \partial_{t} u \Bigr|\, dx, \\
I_{3}(t)  &= 2\int_{\T^3}\Bigl| \rho\,\, u\, \nabla({\partial_{t} u})\, {\partial_{t} u}\Bigr|\, dx,\\
I_{4}(t)  &= \int_{\T^3}\Bigr|\, \rho\, ((u\, \cdotp \nabla u)\,\cdotp \nabla u)\, \cdotp\,  \partial_{t} u\Bigr|\, dx,\\ 
I_{5}(t)  &= \int_{\T^3}\Bigr| \,\rho\, ((u\,\otimes u) : \nabla^{2}) u\,\cdotp \partial_{t} u \Bigr| \, dx,\\
I_{6}(t)  &= \int_{\T^3}\Bigr| \rho\,\, ( u\cdotp \nabla u)\,\cdot (u \cdot \nabla(\partial_{t} u)) \Bigr| \, dx.\\
\end{split}
\end{equation}

\vsd
\noindent As far as $I_{2}(t) $ is concerned, firstly we apply H\"older's inequality and we get  
\begin{equation}
\begin{split}
I_{2}(t)  &= \int_{\T^3}\bigl| \rho \,\, (\partial_{t} u \cdot\nabla{u})\, \partial_{t} u \bigr| \, dx\\
&\leq  M\, \,\|  \partial_{t} u(t) \|_{L^2}\, \|  \partial_{t} u(t) \|_{L^6}\, \|  \nabla u(t) \|_{L^3}.\\
\end{split}
\end{equation}
Once again, classical Sobolev embedding can not be applied directly to the term $\|  \partial_{t} u(t) \|_{L^6}$. We shall consider the term $\overline{\partial_{t} u(t)}$ and adapt Lemma \ref{average}. Firstly, notice that $\dis{\int_{\T^3} \rho(t,x)\, \partial_{t} u(t,x)  \,dx =0}$, due to an integration of the momentum equation in (\ref{system NSI 3})). Hence, \textit{the average method} gives rise to the following computation
\begin{equation*}
\begin{split}
\int_{\T^3} (\rho(t,x) - \bar{\rho}(t)) \, (\partial_{t} u(t,x) - \overline{\partial_{t} u(t)}) \, dx &= \int_{\T^3}   \rho(t,x) \, \partial_{t} u(t,x)\ dx \, \, - \, \bar{\rho}(t)\, \overline{\partial_{t} u(t)}. 
\end{split}
\end{equation*}
\noindent By vertue of remarks \ref{argument 1} and \ref{argument 3}, one has
\begin{equation*}
\vert \,\overline{\partial_{t} u(t)}\vert \, \leqslant\, \frac{1}{\bar{\rho_{0}}} \,  \|  \rho_{0} - \bar{\rho_{0}} \|_{L^2}  \, \|  \partial_{t} u (t) - \overline{\partial_{t} u(t)} \|_{L^2},
\end{equation*}
\noindent which gives, thanks to Poincaré-Wirtinger 
\begin{equation*}
\vert \,\overline{\partial_{t} u(t)} \vert \, \leqslant\, \frac{1}{\bar{\rho_{0}}} \,  \|  \rho_{0} - \bar{\rho_{0}} \|_{L^2}  \, \|  \nabla \partial_{t} u (t) \|_{L^2}.
\end{equation*}
Therefore, we deduce from the above computation that 
$$ \| \partial_{t} u (t)\|_{L^6} \leqslant \| \partial_{t} u (t) - \overline{\partial_{t} u(t)} \|_{L^6} + \, \vert \,\overline{\partial_{t} u(t)}\vert \,  \leqslant  C(\rho_{0})\, \|\nabla {\partial_{t} u(t)}\|_{L^2}.$$ 
\noindent Thanks to Gagliardo-Niremberg and Young inequalities, we infer that 
\begin{equation}
\begin{split}
I_{2}(t) & \leq  \, C(\rho_{0})\, \|  \partial_{t} u(t) \|_{L^2}\, \|  \nabla{\partial_{t} u(t)} \|_{L^2}\, \|  \nabla u(t) \|^{\frac{1}{2}}_{L^2} \|  \nabla^2 u(t) \|^{\frac{1}{2}}_{L^2}\\
&\leq  C(\rho_{0})\,\,  \|  \partial_{t} u(t) \|^2_{L^2}\,  \|  \nabla u(t) \|_{L^2}\, \|  \nabla^2 u(t) \|_{L^2}  + \frac{1}{4}\|  \nabla{\partial_{t} u(t)} \|^2_{L^2}\,\\
&\leq C(\rho_{0})\,\, \|  \partial_{t} u(t) \|^2_{L^2}\, \Bigl( \|  \nabla u(t) \|^2_{L^2}\, +\, \|  \nabla^2 u(t) \|^2_{L^2}\Bigr)  + \frac{1}{4}\|  \nabla{\partial_{t} u(t)} \|^2_{L^2}.
\end{split}
\end{equation}

\noindent Concerning estimate of $I_{3}(t) $, we get
\begin{equation}
\begin{split}
I_{3}(t)  &= \int_{\T^3}\Bigl| \rho\,\, u\, \nabla({\partial_{t} u(t) })\, {\partial_{t} u(t) }\Bigr|\, dx\\
&\leq M\,\, \|  u\, \partial_{t} u \|_{L^2}\,  \|  \nabla{\partial_{t} u}\|_{L^2}\\
&\leq  M\, \,\|  u \|_{L^3}\, \|  \partial_{t} u \|_{L^6}\, \|  \nabla{\partial_{t} u} \|_{L^2}.\\
\end{split}
\end{equation}

\noindent Applying \textit{the average method} for $\|  \partial_{t} u(t)  \|_{L^6} $ and $\|  u(t) \|_{L^3}$, we infer that
\begin{equation}
\begin{split}
I_{3}(t) &\leq    C(\rho_{0})\, \|  u(t)\|_{L^3}\, \|  \nabla{\partial_{t} u(t)} \|^2_{L^2}\\
&\leqslant C(\rho_{0})\,\, \|  \nabla{u}(t) \|_{L^2}\, \|  \nabla{\partial_{t} u}(t)  \|^2_{L^2}.\\
\end{split}
\end{equation}

\vsd
\noindent Concerning $I_{4}(t) , I_{5}(t) $, and $I_{6}(t) $, previous computations hold (applying Proposition \ref{application1 lemme average} and Young inequality) : 
\begin{equation}
\begin{split}
I_{4}(t)  &= \int_{\T^3} \Bigr|\, \rho\, ((u\, \cdotp \nabla u)\,\cdotp \nabla u)\, \cdotp\,  \partial_{t} u\Bigr| \,dx\\
&\leq  M\, \,\|  u(t) \|_{L^6}\, \|  \nabla u(t) \|^2_{L^6}\, \|  \partial_{t} u(t) \|_{L^2}\\
&\leq  C(\rho_{0})\,\, \|  \nabla u(t)\|_{L^2}\, \|  \nabla^2 u(t) \|^2_{L^2}\, \| \partial_{t} u (t) \|_{L^2}\\
&\leq  \frac{1}{4}\,   C(\rho_{0})\,\Bigl( \|  \nabla u(t)\|^2_{L^2}\, +\, \|  \nabla^2 u(t) \|^2_{L^2}  \Bigr) \, \Bigl( \|  \nabla^2 u(t) \|^2_{L^2}\, +\, \,\| \partial_{t} u(t)  \|^2_{L^2}  \Bigr).
\end{split}
\end{equation}

\vsd
\begin{equation}
\begin{split}
I_{5}(t)  &= \int_{\T^3} \Bigr| \,\rho\, ((u\,\otimes u) : \nabla^{2}) u\,\cdotp \partial_{t} u \Bigr|:, dx\\
&\leq M\, \,\|  u^2 \, \partial_{t} u(t) \|_{L^2}\,  \|  \nabla^2 u(t)\|_{L^2}\\
&\leq  C(\rho_{0})\,\, \|  u(t) \|^2_{L^6}\, \|  \nabla{\partial_{t} u(t)} \|_{L^2} \| \nabla^2 u(t)  \|_{L^2}\,  \\
&\leq  C(\rho_{0})\,\, \|  \nabla u(t) \|^2_{L^2}\, \|  \nabla{\partial_{t} u(t)} \|_{L^2} \| \nabla^2 u  \|_{L^2} \\
& \leq   C(\rho_{0})\,\, \|  \nabla u(t) \|^4_{L^2}\, \| \nabla^2 u(t)  \|^2_{L^2} + \frac{1}{4} \|  \nabla{\partial_{t} u(t)} \|^2_{L^2}.
\end{split}
\end{equation}

\vsd
\noindent Similar computation holds for the last term $I_{6}(t)$.
\begin{equation}
\begin{split}
I_{6}(t)  &= \int_{\T^3} \Bigr| \rho\,\, ( u\cdotp \nabla u)\,\cdot (u \cdot \nabla(\partial_{t} u)) \Bigr|\,dx\\
&\leq M\, \,\|  u^2 \, \nabla u(t) \|_{L^2}\,  \| \nabla{\partial_{t} u(t)} \|_{L^2}\\
& \leq C(\rho_{0})\,\,   \|  \nabla u(t) \|^4_{L^2}\, \| \nabla^2 u(t)  \|^2_{L^2} + \frac{1}{4} \|  \nabla{\partial_{t} u(t)} \|^2_{L^2}.
\end{split}
\end{equation}
\vsd

\noindent Let us keep on the proof. Plugging these above estimates into the (\ref{estimatederiv\'ee}) gives rise to 
\begin{equation}
\begin{split}
\frac{m}{2}\, \dfrac{d}{dt}\bigl( \,\| \partial_{t} u \|^2_{L^2}\bigr) + \, \, \|\nabla{\partial_{t} u(t)}\|^2_{L^2}  &\leqslant\, \frac{m}{2}\,\| \partial_{t} u(t) \|^{2}_{L^2}\, + C(\rho_{0})\, \,\|  \partial_{t} u(t) \|^2_{L^2}\, \Bigl( \|  \nabla u(t) \|^2_{L^2}\, +\, \|  \nabla^2 u (t)\|^2_{L^2}\Bigr)\\  &+ \frac{1}{4}\|  \nabla{\partial_{t} u(t)} \|^2_{L^2}\, + C(\rho_{0})\, \|  \nabla{u} (t)\|_{L^2}\, \|  \nabla{\partial_{t} u(t)} \|^2_{L^2}\\
&+ C(\rho_{0})\, \Bigl( \|  \nabla u(t)\|^2_{L^2}\, +\, \|  \nabla^2 u (t)\|^2_{L^2}  \Bigr) \, \Bigl( \|  \nabla^2 u(t) \|^2_{L^2}\, +\, \| \partial_{t} u(t)  \|^2_{L^2}  \Bigr)\\ &+  2\, C(\rho_{0})\, \, \|  \nabla u(t) \|^4_{L^2}\, \| \nabla^2 u(t)  \|^2_{L^2} + \frac{1}{2} \|  \nabla{\partial_{t} u(t)} \|^2_{L^2},\\
\end{split}
\end{equation}
\noindent so that
\begin{equation*}
\begin{split}
\frac{m}{2}\, \dfrac{d}{dt}\bigl( \,\| \partial_{t} u(t) \|^2_{L^2}\bigr) + \, \, \frac{1}{4} \|\nabla{\partial_{t} u(t)}\|^2_{L^2}  &\leqslant\, \frac{m}{2}\,\| \partial_{t} u(t) \|^{2}_{L^2}\,  +   \, C(\rho_{0}) \,\|  \nabla{u}(t)\|_{L^2}\, \|  \nabla{\partial_{t} u(t)} \|^2_{L^2}\\ &+ 2\,C(\rho_{0})\, \, \|  \nabla u(t) \|^4_{L^2}\, \| \nabla^2 u(t)  \|^2_{L^2}\\ &+ C(\rho_{0})\,\Bigl( \|  \nabla u(t)\|^2_{L^2}\, +\, \|  \nabla^2 u(t) \|^2_{L^2}  \Bigr) \, \Bigl( \|  \nabla^2 u(t) \|^2_{L^2}\, +\, \| \partial_{t} u(t)  \|^2_{L^2}  \Bigr). 
\end{split}
\end{equation*}
By integration in time, we have : 
\begin{equation}
\label{eq integr\'ee}
\begin{split}
\frac{m}{2}\,\, \| \partial_{t} u(t) \|^2_{L^2} + &\frac{1}{2} \int_{0}^{t} \,\|\nabla{\partial_{t} u(t')}\|^2_{L^2} \, dt'  \leqslant\, \| u_{0} \|^2_{H^2} \, +  \frac{m}{2}\,\int_{0}^{t}\| \partial_{t} u (t')\|^{2}_{L^2}\, dt'\\ & +\, C(\rho_{0})\, \int_{0}^{t}\,\|  \nabla{u}(t')\|_{L^2}\, \|  \nabla{\partial_{t} u(t')} \|^2_{L^2} \, dt' \\  &+ C(\rho_{0})\, \int_{0}^{t} \, \|  \nabla u(t') \|^4_{L^2}\, \| \nabla^2 u(t')  \|^2_{L^2} \, dt'\\ &+C(\rho_{0})\, \int_{0}^{t} \,\Bigl( \|  \nabla u(t')\|^2_{L^2}\, +\, \|  \nabla^2 u(t') \|^2_{L^2}  \Bigr) \, \Bigl( \|  \nabla^2 u(t') \|^2_{L^2}\, +\, \| \partial_{t} u(t')  \|^2_{L^2}  \Bigr)\, dt'.\\
\end{split}
\end{equation}
Concerning the term $\dis{\int_{0}^{t}\,\|  \nabla{u}(t')\|_{L^2}\, \|  \nabla{\partial_{t} u(t')} \|^2_{L^2} \, dt'}$  
$$\int_{0}^{t}\,\|  \nabla{u}(t')\|_{L^2}\, \|  \nabla{\partial_{t} u(t')} \|^2_{L^2} \, dt' \leq \,  \|  \nabla{u}\|_{L^{\infty}_{T}(L^2)}\, \int_{0}^{t} \,\|  \nabla{\partial_{t} u(t')} \|^2_{L^2} \,dt',$$
\noindent which becomes, by vertue of Theorem \ref{theorem1GWP}, 
$$\int_{0}^{t} \,\|  \nabla{u}(t')\|_{L^2}\, \|  \nabla{\partial_{t} u(t')} \|^2_{L^2} \, dt' \leq C\,  \|  \nabla{u_{0}}\|_{L^2}\, \int_{0}^{t} \, \|  \nabla{\partial_{t} u(t')} \|^2_{L^2}\, dt'.$$
\noindent Same argument combining with Theorem \ref{theorem1GWP} gives rise to
$$ \int_{0}^{t} \, \|  \nabla u(t') \|^4_{L^2}\, \| \nabla^2 u(t')  \|^2_{L^2} \, dt' \leq C\, \|  \nabla{u_{0}}\|^6_{L^2}.$$
\noindent As a result, Inequation (\ref{eq integr\'ee}) can be rewritten as follows ( providing we choose $\|  \nabla{u_{0}}\|_{L^2}$ small enough)  
\begin{equation}
\label{sum1}
\begin{split}
\frac{m}{2}\, \, \| \partial_{t} u(t) \|^2_{L^2} +& \frac{1}{3} \int_{0}^{t}  \,\|\nabla{\partial_{t} u(t')}\|^2_{L^2} \, dt'  \lesssim\, \| u_{0} \|^2_{H^2}  + \, \frac{m}{2}\,\|  \nabla{u_{0}}\|^2_{L^2} + \|  \nabla{u_{0}}\|^6_{L^2} \,\\ &+\,\int_{0}^{t}  \,\Bigl( \|  \nabla u(t')\|^2_{L^2}\, +\, \|  \nabla^2 u(t') \|^2_{L^2}  \Bigr) \, \Bigl( \|  \nabla^2 u(t') \|^2_{L^2}\, +\, \| \partial_{t} u(t')  \|^2_{L^2}  \Bigr)\, dt'. 
\end{split}
\end{equation}
\noindent Moreover, the momentum equation given by
$$ - \Delta{u} + \nabla{\Pi} = \,-\,\rho\,\bigl(\partial_{t} u +  u\cdot\nabla{u} \bigr),$$
which along with the classical estimates on the Stokes system, ensures that
\begin{equation*}
\begin{split}
 \|\nabla^2{u}(t)\|_{L^2} + \| \nabla{\Pi}(t)\|_{L^2} &\leqslant C\, \Bigl( \| \partial_{t} u(t)\|_{L^2} + \|  \nabla u(t)\|^{\frac{3}{2}}_{L^2} \| \nabla^2{u}(t)\|^{\frac{1}{2}}_{L^2}\Bigr)  \\
&\lesssim  \| \partial_{t} u(t)\|_{L^2} + \|  \nabla u(t)\|^{3}_{L^2} + \frac{1}{2}\| \nabla^2{u}(t)\|_{L^2}.\\
\end{split}
\end{equation*}\noindent So that, we get
\begin{equation}
\begin{split}
\frac{1}{2} \, \|\nabla^2{u}(t)\|_{L^2} + \, \| \nabla{\Pi}(t)\|_{L^2} &\lesssim  \,\| \partial_{t} u(t)\|_{L^2} +  \,\|  \nabla u(t)\|^{3}_{L^2}.
\end{split}
\end{equation}
\noindent By vertue of Theorem \ref{theorem1GWP}, we obtain 
\begin{equation}
\label{stokes regularity}
\sup_{t \in[0,T]} \Bigl(\frac{1}{2} \, \|\nabla^2{u}(t)\|^2_{L^2} + \,\| \nabla{\Pi}(t)\|^2_{L^2}\Bigr) \lesssim \sup_{t \in[0,T]} \bigl(\,\| \partial_{t} u(t)\|^2_{L^2} \bigr)+ \|  \nabla u_{0}\|^{6}_{L^2}. 
\end{equation}

\vskip 0.2cm
\begin{remark}
Let us point out that searching an estimate of $\| u\|_{L^2_{T}(H^3)}$ is a natural idea here since the initial velocity $u_{0}$ belongs to the space $H^2$. But actually, it is not relevant. Indeed, to perform it, we shall use the theory of Stokes problems. We shall begin derivating the momentum equation with respect to the space, and then, we shall take the $L^2$ norm. But, such an approach is doomed to fail, because requires an estimate on  $\dis{\sup_{t \in [0,T]} \| \nabla  \rho\|_{L^\infty}}$, which is not our case here, since the density function only belongs to $L^{\infty}([0,T] \times \T^3)$.
\end{remark}

\noindent Once again, the momentum equation gives  
$$ -\Delta{u} + \nabla{\Pi} = \,-\,\rho(\partial_{t} u + u\cdot\nabla{u}).$$
We take the $L^6$-norm and use the fact that $\|u\cdot\nabla{u}(t) \|_{L^6} \leqslant C\, \|\nabla (u\cdot\nabla{u}(t)) \|_{L^2} $ since $\overline{u\cdot\nabla{u}}=0$.  
\begin{equation*}
\begin{split}
\|\nabla^2{u}(t)\|_{L^6} \,+ \,\| \nabla^2 p(t) \|_{L^6} &\leqslant \| {\rho \bigl(\partial_{t} u\, +\, u\cdot\nabla{u}}\bigr)\|_{L^6} \\ 
&\leq  C(\rho_{0})\,  \Bigl(  \,\,\| \nabla{\partial_{t} u(t)} \|_{L^2}\,\, + \| (\nabla u(t))^2\|_{L^2} + \|  u(t) (\nabla^2 u(t))\|_{L^2}\Bigr)\\
&\leqslant C(\rho_{0})\,   \Bigl(  \,\,\| \nabla{\partial_{t} u(t)} \|_{L^2}\,\, + \| \nabla u(t)\|_{L^3}\, \| \nabla^2 u(t)\|_{L^2} \, + \,  \|  u(t)\|_{L^3}\, \| \nabla^2 u(t)\|_{L^6}\Bigr)
\end{split}
\end{equation*}
\noindent Applying Proposition \ref{application 2 lemme average} to the term $\| \nabla u(t)\|_{L^3}$, we get 
\begin{equation*}
\begin{split}
\|\nabla^2{u}(t)\|_{L^6} \,+ \,\| \nabla^2 p(t) \|_{L^6} &\lesssim   \,\| \nabla{\partial_{t} u(t)} \|_{L^2}\,\, + \| \nabla u(t)\|^{\frac{1}{2}}_{L^2}\,\| \nabla^2 u(t)\|^{\frac{3}{2}}_{L^2} \, + \,  \|  \nabla u(t)\|_{L^2}\, \| \nabla^2 u(t)\|_{L^6}.
\end{split}
\end{equation*}
\noindent By integration in time : 
\begin{equation*}
\begin{split}
\int_{0}^{t}\,\|\nabla^2{u}(t')\|^2_{L^6} \, dt' \,+ \,\int_{0}^{t}\,\| \nabla^2 p(t') \|^2_{L^6}\, dt' &\lesssim \,\int_{0}^{t}\,\| \nabla{\partial_{t} u(t')} \|^2_{L^2} \, dt'\,\, + \int_{0}^{t}\,\| \nabla u(t')\|_{L^2}\,\| \nabla^2 u(t')\|^{3}_{L^2} \,dt' \\ &+ \,  \|  \nabla u\|^2_{L^{\infty}_{T}(L^2)}\, \int_{0}^{t} \,\| \nabla^2 u(t')\|^2_{L^6}\, dt'.
\end{split}
\end{equation*}
\vsd
\noindent On the one hand, Theorem \ref{theorem1GWP} provides $\ds{\|  \nabla u\|^2_{L^{\infty}_{T}(L^2)} \lesssim \|  \nabla u_{0}\|^2_{L^2}}$, which implies that $$\|  \nabla u\|^2_{L^{\infty}_{T}(L^2)}\, \int_{0}^{t} \,\| \nabla^2 u(t')\|^2_{L^6} \, dt' \leqslant \|  \nabla u_{0}\|^2_{L^2}\, \int_{0}^{t} \,\| \nabla^2 u(t')\|^2_{L^6} \, dt'.$$ 
\noindent On the other hand, applying Estimates (\ref{L2 energy estimate}) and (\ref{H1 energy estimate}) of Theorem \ref{theorem1GWP}, to the term $$\int_{0}^{t} \,\| \nabla u(t')\|_{L^2}\, \| \nabla^2 u(t')\|^{3}_{L^2}\, dt' ,$$ leads to 
\begin{equation*}
\begin{split} 
\int_{0}^{t} \, \| \nabla u(t')\|_{L^2}\, \| \nabla^2 u(t')\|^{3}_{L^2} \, dt' &= \int_{0}^{t}\, \| \nabla u(t')\|_{L^2}\, \| \nabla^2 u(t')\|_{L^2}\,\, \| \nabla^2 u(t')\|^{2}_{L^2}\, dt'\\
&\leq\, \sup_{t \in[0,T]} \bigl(\,\| \nabla^2{u}(t) \|^2_{L^2}\bigr)\,\,  \Bigl(\int_{0}^{t} \|  \nabla u(t')\|^2_{L^2}\, dt'\Bigr)^{\frac{1}{2}}\,  \Bigl(\int_{0}^{t} \| \nabla^2{u}(t') \|^2_{L^2}\, dt'\Bigr)^{\frac{1}{2}}\\
& \lesssim  \| u_{0}\|_{L^2}\, \| \nabla u_{0}\|_{L^2}\,\,  \sup_{t \in[0,T]} \bigl(\,\| \nabla^2{u}(t) \|^2_{L^2}\bigr)\,\,
\end{split}\end{equation*}
\vsd
\noindent As a result, if $\| \nabla u_{0}\|_{L^2}$ is small enough, we have : 
 \begin{equation}
 \label{regularity stokes H2}
\begin{split}
\frac{\mu}{2} \int_{0}^{t} \,\|\nabla^2{u}(t')\|^2_{L^6} \, dt'\,+ \,\int_{0}^{t} \,\| \nabla^2 \Pi (t')\|^2_{L^6}\, dt' &\lesssim  \, \| u_{0}\|_{L^2}\, \| \nabla u_{0}\|_{L^2}\,\,  \sup_{t \in[0,T]} \bigl(\,\| \nabla^2{u}(t) \|^2_{L^2}\bigr)\\ &+ \,\int_{0}^{t} \,\| \nabla{\partial_{t} u(t')} \|^2_{L^2} \, dt'.
\end{split}
\end{equation}
\noindent Summing (\ref{regularity stokes H2}) with (\ref{stokes regularity}) and (\ref{sum1}), we recognize $B_{2}(T)$ and we get
\begin{equation*}
\begin{split}
B_{2}(T) &\lesssim  \frac{m}{2}\,\|  \nabla{u_{0}}\|^2_{L^2}  +\,  \|  \nabla{u_{0}}\|^6_{L^2} \, + \| u_{0} \|^2_{H^2} \,\\    &\qquad  +\frac{1}{4}\int_{0}^{t} \Bigl( \|  \nabla u(t')\|^2_{L^2}\, +\, \|  \nabla^2 u(t') \|^2_{L^2} \Bigr)\, dt'  \, \Bigl( \sup_{t \in[0,T]} \,\|  \nabla^2 u(t) \|^2_{L^2}\, +\,\sup_{t \in[0,T]} \,\| \partial_{t} u(t)  \|^2_{L^2}  \Bigr)\\ &\qquad  + \int_{0}^{t} \,\| \nabla{\partial_{t} u}(t') \|^2_{L^2} \, dt'\,\, + \| u_{0}\|_{L^2}\, \| \nabla u_{0}\|_{L^2}\,\,  \sup_{t \in[0,T]} (\, \| \nabla^2{u}(t) \|^2_{L^2})\,\,
\end{split}
\end{equation*}
The smallness condition on $\| u_{0}\|_{L^2}\, \| \nabla u_{0}\|_{L^2} $ implies 
\begin{equation*}
\begin{split}
B_{2}(T) &\lesssim  \frac{m}{2}\,\|  \nabla{u_{0}}\|^2_{L^2}  +\,  \|  \nabla{u_{0}}\|^6_{L^2} +\, \| u_{0} \|^2_{H^2} \, \\ &\qquad \qquad +\frac{1}{4}\int_{0}^{t} \Bigl( \|  \nabla u(t')\|^2_{L^2}\, +\, \|  \nabla^2 u(t') \|^2_{L^2} \Bigr)dt'  \, \Bigl( \sup_{t \in[0,T]} \,\|  \nabla^2 u(t) \|^2_{L^2}\, +\,\sup_{t \in[0,T]} \,\| \partial_{t} u(t)  \|^2_{L^2}  \Bigr).
\end{split}
\end{equation*} 
Now, we apply Gronwall lemma, and we have : 
\begin{equation}
\begin{split}
B_{2}(T) &\lesssim \Bigl(  \frac{m}{2}\,\|  \nabla{u_{0}}\|^2_{L^2}  +\,  \|  \nabla{u_{0}}\|^6_{L^2} \ + \, \| u_{0} \|^2_{H^2} \,\Bigr)\,\,
\exp{\Bigl(\int_{0}^{t} \|  \nabla u(t')\|^2_{L^2}\, dt'\,+\, \|  \nabla^2 u(t') \|^2_{L^2} \, dt' \Bigr)}.
\end{split}
\end{equation}
\noindent Once again Theorem \ref{theorem1GWP} gives the expected estimate in the exponential term. Finally, we get
\begin{equation}
\begin{split}
B_{2}(T)
&\lesssim    \bigl(  1 +\,  \|  u_{0}\|^4_{H^2} \,\bigr) \, \| u_{0} \|^2_{H^2} \,\
\exp{\Bigl(  \|  u_{0}\|^2_{L^2} + \|  \nabla u_{0}\|^2_{L^2} \Bigr)}. 
\end{split}
\end{equation}
\noindent This concludes the proof of \ref{H2 energy estimate}. Up to the regularization procedure of Friedrich, we have proved the global existence of solution of \ref{system NSI 3}, with data $(\rho_{0},u_{0})$ satisfiying hypothesis of Theorem \ref{theorem1GWP}.

\vskip 0.3cm

\section{Appendix}
\begin{lemma}{(Gronwall's Lemma)}\\
\label{gronwall}\sl{
Let $f$ and $g$ be two positive functions satisfying $\ds{\frac{1}{2}\dfrac{d}{dt} f^{2}(t) \leqslant f(t)\,g(t)}$.
Then, we have $$f(t) \leqslant f(0) + \int_{0}^{t} g(t')dt'.$$
}\end{lemma}

\begin{proof}
\noindent We introduce the function $\ds{H(t) \eqdefa 2\,\int_{0}^{t} f(t')\,g(t')\, dt'}$. As defined, we get immediately 
\begin{equation}
H'(t) =2\,f(t)\, g(t) \quad \hbox{and} \quad f^{2}(t)-f^{2}(0) \leqslant H(t).
\end{equation}
\noindent This implies that for any $\varepsilon >0$,
$$ f(t) \leqslant \sqrt{H(t) + f^{2}(0) + \varepsilon^2}.$$
\noindent Moreover, we have in particular $\ds{H'(t) \leqslant 2\,\sqrt{H(t) + f^{2}(0) + \varepsilon^2}\,\,  g(t)}$ and thus $$\dfrac{d}{dt} \sqrt{H(t) + f^{2}(0) + \varepsilon^2} \leqslant g(t).$$
\noindent By integration in time, we have 
$$ \sqrt{H(t) + f^{2}(0) + \varepsilon^2} \leqslant \sqrt{H(0) + f^{2}(0) + \varepsilon^2} + \int_{0}^{t} g(t')\,dt'.$$ 
\noindent Finally, we have for any $\varepsilon >0$,
$$ f(t) \leqslant \sqrt{f^2(0) + \varepsilon^2 } + \int_{0}^{t} g(t')\,dt',$$
\noindent which proves the result.
\end{proof}

\vskip 0.5cm
\begin{lemma}
\label{lemma product law}\sl{
\noindent The following properties hold
\begin{enumerate}
\item Sobolev embedding: if $p_{1} \leq p_{2}$ and $r_{1} \leq r_{2}$, then\\
$$ B^{s}_{p_{1},r_{1}} \hookrightarrow B^{s - N(\frac{1}{p_{1}} - \frac{1}{p_{2}})}_{p_{2},r_{2}}.$$
\item Product laws in Besov spaces: let $1 \leqslant r,p,p_{1},p_{2} \leqslant +\infty$.\\ If $s_{1},s_{2} < \frac{N}{p}$ and $s_{1}+s_{2} +N\,\min(0,1-\frac{2}{p}) >0$, then 
$$ \| uv\|_{B^{s_{1}+s_{2}- \frac{N}{p}}_{p,r}} \leqslant C\, \| u\|_{B^{s_{1}}_{p,r}} \, \| v\|_{B^{s_{2}}_{p,\infty}}.$$ 
\item Another product law: if $|s| < \frac{N}{p}$, then
$$ \| uv\|_{B^{s}_{p,r}} \leqslant C\, \| u\|_{B^{s}_{p,r}}\, \| v\|_{B^{\frac{N}{p}}_{p,\infty} \cap L^{\infty}}.$$
\item Algebric properties: for $s>0$, $B^{\frac{N}{p}}_{p,\infty} \cap L^{\infty}$ is an algebra. Moreover, for any $p \in [1,+\infty$, then 
$$ B^{\frac{N}{p}}_{p,1} \hookrightarrow B^{\frac{N}{p}}_{p,\infty} \cap L^{\infty}.$$
\end{enumerate}
}\end{lemma}
\vskip 0.5cm

\begin{lemma}\sl{
\label{derivée dans les besov}
\noindent Let $\mathcal{C}$ a ring of $\R^3$. A constant $C$ exists so that for any positive real number $\lambda$, any non-negative integer $k$, the following hold\\
$$\hbox{If} \quad \hbox{Supp}\,\,\widehat{u}\subset \lambda\, \mathcal{C},\quad then \quad C^{-1-k}\, \lambda^k\, \| u\|_{L^a}\,\, \leqslant\,\, \sup_{|\alpha| = k}\, \| \partial^{\alpha}u\|_{L^a}\,\, \leqslant\,\, C^{1+k}\, \lambda^k\, \| u\|_{L^a}.$$
}\end{lemma}

\begin{lemma}\sl{
\begin{equation}
\sum_{q \in \Z} 2^{\frac{3q}{2}}\, \bigl\| \left[ \Delta_{q}, S_{m}a \right]\, \nabla u  \bigr\|_{L^1_{t}(L^2)} \leqslant C\,\Bigl( 2^{m}\, \|a\|_{L^{\infty}_{t}(L^\infty)} \, \|u \|_{L^1_{t}(B^{\frac{3}{2}}_{2,1})} \, + \, 2^{2m}\, \|a\|_{L^{\infty}_{t}(L^2)} \, |u \|_{L^1_{t}(\dot{H}^{2})}  \Bigr)
\end{equation}
}\end{lemma}

\begin{proof}
By vertue of Bony's decomposition, the commutator may be decomposed into 
\begin{equation}
\begin{split}
\left[ \Delta_{q}, S_{m}a \right] &= \Delta_{q}(S_{m}a\, \nabla u) - S_{m}a \Delta_{q} \nabla u\\
&= \Delta_{q}(T_{S_{m}a}\nabla u) + \Delta_{q}(T_{\nabla u}S_{m}a) + \Delta_{q} R(S_{m}a, \, \nabla u)\\ &\qquad - T_{S_{m}a} \Delta_{q}\nabla u - T_{\Delta_{q}\nabla u} S_{m}a - R(S_{m}a\,,\,\Delta_{q}\nabla u)\\
&= \left[ \Delta_{q}, T_{S_{m}a} \right]\nabla u \,+\, \Delta_{q}(T_{\nabla u}S_{m}a) +\, \Delta_{q}R(S_{m}a, \, \nabla u) \, -\, T^{'}_{\Delta_{q}\nabla u} S_{m}a, 
\end{split}
\end{equation}
\noindent where $\ds{T'_{a}b\, \eqdefa\, T_{a}b\, +\, R(a,b)}$. 
\noindent Let us analyse each term in the right-hand-side. Firstly, we decompose the first commutator term into
\begin{equation}
\begin{split}
\left[ \Delta_{q}, T_{S_{m}a} \right]\nabla u &= \Delta_{q}(T_{S_{m}a}\, \nabla u) - T_{S_{m}a} \Delta_{q} \nabla u\\
&= \Delta_{q} \Bigl(\sum_{|q-q'|\leqslant 4} S_{q'-1}S_{m}a\, \Delta_{q'}\nabla u \Bigr) - \sum_{|q-q'|\leqslant 4} S_{q'-1}S_{m}a\, \Delta_{q'} \Delta_{q} \nabla u\\
&= \sum_{|q-q'|\leqslant 4} \left[ \Delta_{q}, S_{q'-1}S_{m}a\, \right]\Delta_{q'}\nabla u.  
\end{split}
\end{equation}
\noindent Now, let us focus on the commutator term $\ds{\left[ \Delta_{q}, S_{q'-1}S_{m}a\, \right]\Delta_{q'}\nabla u}$. We shall use definiton of Littlewood-Paley theory. 
\begin{equation}
\begin{split}
\left[ \Delta_{q}, S_{q'-1}S_{m}a\, \right]\Delta_{q'}\nabla u &= \Delta_{q} \Bigl( S_{q'-1}S_{m}a\, \Delta_{q'}\nabla u  \Bigr) - S_{q'-1}S_{m}a\, \Delta_{q} \Delta_{q'}\nabla u\\
&= \varphi(2^{-q}\,|D|)\,S_{q'-1}S_{m}a\, \Delta_{q'}\nabla u - S_{q'-1}S_{m}a\,\varphi(2^{-q}\,|D|)\, \Delta_{q'}\nabla u.   
\end{split}
\end{equation}
\noindent In particular, writting $\ds{h \eqdefa \mathcal{F}^{-1}\varphi(\vert \cdotp \vert)}$, we get 
\begin{equation}
\begin{split}
\varphi(2^{-q}\,|D|)\,S_{q'-1}S_{m}a\, \Delta_{q'}\nabla u(x) &\eqdefa \int_{\T^3} 2^{qd}\, h(2^{q}y)\, S_{q'-1}S_{m}a(x-y)\, \Delta_{q'}\nabla u(x-y)\, dy \\
&= \int_{\T^3}  h(z)\, S_{q'-1}S_{m}a(x-2^{-q}z)\, \Delta_{q'}\nabla u(x-2^{-q}z)\, dz.    
\end{split}
\end{equation}
\noindent Likewise, we have
\begin{equation}
\begin{split}
S_{q'-1}S_{m}a\,\varphi(2^{-q}\,|D|)\, \Delta_{q'}\nabla u (x) &\eqdefa S_{q'-1}S_{m}a\,(x)\, \int_{\T^3} 2^{qd}\, h(2^{q}y)\, \Delta_{q'}\nabla u(x-y)\, dy\\ 
&= \int_{\T^3}  S_{q'-1}S_{m}a\,(x)\,\,  h(z)\, \Delta_{q'}\nabla u(x-2^{-q}z)\, dz.    
\end{split}
\end{equation}
\noindent Therefore, applying the first-order Taylor's formula, we get, for any $x \in \T^3$,  
\begin{equation}
\begin{split}
\left[ \Delta_{q}, S_{q'-1}S_{m}a\, \right]\Delta_{q'}\nabla u(x) &= \int_{\T^3}  h(z)\, \left[ S_{q'-1}S_{m}a(x-2^{-q}z)\, - S_{q'-1}S_{m}a\,(x)  \right]\, \Delta_{q'}\nabla u(x-2^{-q}z)\, dz\\
&= -\,\int_{\T^3} \int_{0}^{1}\, h(z)\,2^{-q}z\,\cdotp \nabla S_{q'-1}S_{m}a(x-2^{-q}\,z\,t)\, \Delta_{q'}\nabla u(x-2^{-q}z)\, dz\,dt\\ 
&=  - 2^{-q}\, \,\int_{\T^3} \int_{0}^{1}\, 2^{qd}\, h(2^{q}y)\, y\cdotp\nabla S_{q'-1}S_{m}a(x-y\,t)\, \Delta_{q'}\nabla u(x-y)\, dz\,dt.\\    
\end{split}
\end{equation}
\noindent Therefore, we infer that, for any $x \in \T^3$,
\begin{equation*}
\| \left[ \Delta_{q}, S_{q'-1}S_{m}a\, \right]\Delta_{q'}\nabla u  \|_{L^2} \leqslant \| \nabla S_{q'-1}S_{m}a  \|_{L^{\infty}} \, 2^{-q}\, \Bigl\| \int_{\T^3}\, 2^{qd}\, (2^{q}y)\, h(2^{q}y)\, \Delta_{q'}\nabla u(\cdotp-y)\, dz   \Bigr\|_{L^2}.
\end{equation*}
\noindent Applying Young's inequality ($\ds{ L^1 \ast L^2 = L^2}$), we infer that 
\begin{equation}
\begin{split}
\| \left[ \Delta_{q}, S_{q'-1}S_{m}a\, \right]\Delta_{q'}\nabla u  \|_{L^2} \leqslant C\, \| \nabla S_{q'-1}S_{m}a  \|_{L^{\infty}} \, 2^{-q}\, \bigl\| \Delta_{q'}\nabla u\bigr\|_{L^2}.
\end{split}
\end{equation}
\noindent Obviously, we have 
$$ \| \nabla S_{q'-1}S_{m}a  \|_{L^{\infty}} \leqslant \| \nabla S_{m}a  \|_{L^{\infty}} \leqslant 2^{m}\, \|a \|_{L^{\infty}}.$$
\noindent Finally, we get 
\begin{equation*}
\| \left[ \Delta_{q}, S_{q'-1}S_{m}a\, \right]\Delta_{q'}\nabla u  \|_{L^2} \leqslant C\,  2^{-q}\, 2^{m}\, \|a \|_{L^{\infty}}\, \| \Delta_{q'}\nabla u  \|_{L^2},
\end{equation*}
\noindent and thus, 
\begin{equation}
\begin{split}
\|  \left[ \Delta_{q}, T_{S_{m}a} \right]\nabla u \|_{L^2} &\leqslant C\, \sum_{|q-q'|\leqslant 4} 2^{-q}\, 2^{m}\, \|a \|_{L^{\infty}}\, \| \Delta_{q'}\nabla u  \|_{L^2}.\\ 
\end{split}
\end{equation}
\noindent As a consequence, we have 
\begin{equation}
\begin{split}
2^{\frac{3q}{2}}\, \|  \left[ \Delta_{q}, T_{S_{m}a} \right]\nabla u \|_{L^2} &\leqslant C\, \sum_{|q-q'|\leqslant 4} 2^{\frac{3q}{2}}\, 2^{-q}\, 2^{m}\, \|a \|_{L^{\infty}}\, 2^{\frac{-q'}{2}}\, 2^{\frac{q'}{2}}\, \| \Delta_{q'}\nabla u  \|_{L^2}\\ 
&\leqslant C\, 2^{m}\, \|a \|_{L^{\infty}}\, \sum_{|q-q'|\leqslant 4} 2^{\frac{q-q'}{2}}\, 2^{\frac{q'}{2}}\, \| \Delta_{q'}\nabla u  \|_{L^2}.\\  
\end{split}
\end{equation}
\noindent By definition of the Besov norm, there exists a serie $(c_{q'})_{q \in \Z}$ belonging to $\ell^{1}(\Z)$ such that 
$$ 2^{\frac{q'}{2}}\, \| \Delta_{q'}\nabla u  \|_{L^2} \leqslant C\, c_{q'}\, \| \nabla u \|_{B^{\frac{1}{2}}_{2,1}}.$$ 
\noindent And thus, 
\begin{equation}
\begin{split}
2^{\frac{3q}{2}}\, \|  \left[ \Delta_{q}, T_{S_{m}a} \right]\nabla u \|_{L^2}
&\leqslant C\, 2^{m}\, \|a \|_{L^{\infty}}\, \| \nabla u \|_{B^{\frac{1}{2}}_{2,1}}\, \sum_{|q-q'|\leqslant 4} 2^{\frac{q-q'}{2}}\, c_{q'}.\\  
\end{split}
\end{equation}
\noindent We notice, by vertue of Young's inequality, that the term $\ds{\sum_{|q-q'|\leqslant 4} 2^{\frac{q-q'}{2}}\, c_{q'}}$ belongs to $\ell^{1}(\Z)$. Indeed, let us define $\ds{d_{q} \eqdefa \sum_{|q-q'|\leqslant 4} 2^{\frac{q-q'}{2}}\, c_{q'} }$. Thanks to Young's inequality, we get
$$ \| d_{q}\|_{\ell^{1}(\Z)} \leqslant \| c_{q}\|_{\ell^{1}(\Z)} \, \times \, \sum_{-4 \leqslant k \leqslant 4} 2^{\frac{k}{2}} \leqslant C.$$   
\noindent Finally, we get 
\begin{equation}
\begin{split}
\sum_{q \in \Z} 2^{\frac{3q}{2}}\, \|  \left[ \Delta_{q}, T_{S_{m}a} \right]\nabla u \|_{L^2}
&\leqslant C\,  2^{m}\, \|a \|_{L^{\infty}}\, \| \nabla u \|_{B^{\frac{1}{2}}_{2,1}}\,\sum_{q \in \Z} d_{q}\\
&\leqslant C\,  2^{m}\, \|a \|_{L^{\infty}}\, \| \nabla u \|_{B^{\frac{1}{2}}_{2,1}}.
\end{split}
\end{equation}
\noindent By integration in time, we infer that
 \begin{equation}
 \label{estimatebis1}
\sum_{q \in \Z} 2^{\frac{3q}{2}}\,  \| \left[ \Delta_{q}, T_{S_{m}a} \right]\nabla u \|_{L^{1}_{t}(L^2)}
\leqslant C\,  2^{m}\, \|a \|_{L^{\infty}_{t}(L^{\infty})}\, \| \nabla u \|_{L^{1}_{t}(B^{\frac{1}{2}}_{2,1})}.
\end{equation}
\noindent This gives the first term in the Lemma.  The second term will stem from remainder terms in the Bony's decomposition. More precisely, concerning the term $\ds{\sum_{q \in \Z} 2^{\frac{3q}{2}}\, \Delta_{q} T_{\nabla u}S_{m}a \|_{L^{1}_{t}(L^2)}}$, we have by definition 
$$ \sum_{q \in \Z} 2^{\frac{3q}{2}}\, \|  \Delta_{q} T_{\nabla u}S_{m}a \|_{L^{1}_{t}(L^2)} \eqdefa \| T_{\nabla u}S_{m}a \|_{B^{\frac{3}{2}}_{2,1}}.$$ 
\noindent By vertue of Theorem $2.82$ in the book \cite{BCDbis}, we have
\begin{equation}
\| T_{\nabla u}S_{m}a \|_{B^{\frac{3}{2}}_{2,1}} \leqslant C\, \| \nabla u\|_{B^{-\frac{1}{2}}_{\infty,2}} \, \| S_{m}a \|_{B^{2}_{2,2}}.
\end{equation}
Moreover, Bernstein result implies the following embedding $\ds{ B^{1}_{2,2} \hookrightarrow B^{-\frac{1}{2}}_{\infty,2} }$. Therefore, we have
\begin{equation}
\| T_{\nabla u}S_{m}a \|_{B^{\frac{3}{2}}_{2,1}} \leqslant C\, \| \nabla u\|_{B^{1}_{2,2} \equiv H^1} \,\, \| S_{m}a \|_{B^{2}_{2,2}}.
\end{equation}
\noindent Applying Poincaré-Wirtinger to $\ds{\| \nabla u\|_{ H^1}}$, (since the average of $\nabla u$ is nul), we infer that the norms $\ds{\| \nabla u\|_{ H^1}}$ and $\ds{\| \nabla u\|_{ \dot{H}^1}}$ are equivalent and thus 
\begin{equation}
\| T_{\nabla u}S_{m}a \|_{B^{\frac{3}{2}}_{2,1}} \leqslant C\, \|  u\|_{\dot{H}^2} \,\, \| S_{m}a \|_{B^{2}_{2,2}}.
\end{equation}
\noindent On the other hand, it seems obvious that $\ds{\| S_{m}a \|_{B^{2}_{2,2}} \leqslant \| S_{m}a \|_{\dot{B}^{2}_{2,2}} \leqslant \| S_{m}a \|_{\dot{H}^2}}$. As a result, 
\begin{equation}
\| T_{\nabla u}S_{m}a \|_{B^{\frac{3}{2}}_{2,1}} \leqslant C\, 2^{2m}\, \|  u\|_{\dot{H}^2} \,\, \| S_{m}a \|_{\dot{H}^2}.
\end{equation}
\noindent Finally, by integration in time and by definition of $\ds{S_{m}a}$, we get 
\begin{equation}
\label{estimatebis2}
\| T_{\nabla u}S_{m}a \|_{L^{1}_{t}(B^{\frac{3}{2}}_{2,1})} \eqdefa \sum_{q \in \Z} 2^{\frac{3q}{2}}\, \|  \Delta_{q} T_{\nabla u}S_{m}a \|_{L^{1}_{t}(L^2)} \leqslant C\, 2^{2m}\, \|  u\|_{L^{1}_{t}(\dot{H}^2)} \,\, \| a \|_{L^{\infty}_{t}(L^2)}.
\end{equation}
\noindent The estimate on the term $\ds{\sum_{q \in \Z} 2^{\frac{3q}{2}}\, \| \Delta_{q}R(S_{m}a, \, \nabla u)\|_{L^{1}_{t}(L^2)}}$ is close to the previous one, by vertue of Theorem page $2.85$ in \cite{BCDbis}. We recall it below.\\

\textit{Remind: If $s_{1}$ and $s_{2}$ are two real numbers, such that $s_{1} + s_{2} >0$, then $$ \|R(u,v) \|_{B^{s_{1} + s_{2}}_{p,r}} \leqslant C(s_{1},s_{2}) \, \| u\|_{B^{s_{1}}_{p_{1},r_{1}}} \, \| v\|_{B^{s_{2}}_{p_{2},r_{2}}}, \quad \hbox{} \quad \frac{1}{p} \eqdefa \frac{1}{p_{1}} + \frac{1}{p_{2}} \quad \hbox{and} \quad \frac{1}{r} \eqdefa \frac{1}{r_{1}} + \frac{1}{r_{2}}\cdotp $$}
\noindent Therefore, we have 
\begin{equation}
\label{estimatebis3}
\sum_{q \in \Z} 2^{\frac{3q}{2}}\, \|  \Delta_{q}R(S_{m}a, \, \nabla u)\|_{L^{1}_{t}(L^2)} \eqdefa \| R(S_{m}a, \, \nabla u)\|_{L^{1}_{t}(B^{\frac{3}{2}}_{2,1})} \leqslant C\, 2^{2m}\, \|  u\|_{L^{1}_{t}(\dot{H}^2)} \,\, \| a \|_{L^{\infty}_{t}(L^2)}. 
\end{equation}

\noindent Concerning the last term,  $\ds{\sum_{q \in \Z} 2^{\frac{3q}{2}}\, \|T^{'}_{\Delta_{q}\nabla u} S_{m}a\|_{L^{1}_{t}(L^2)}}$, we write the definition. Indeed, 
$$  T^{'}_{\Delta_{q}\nabla u} S_{m}a\, \eqdefa\, \sum_{q' \geqslant q-2} S_{q'+2}\Delta_{q}\nabla u \, \Delta_{q'}S_{m}a.$$
Therefore, we get 
\begin{equation}
\begin{split}
\|  T^{'}_{\Delta_{q}\nabla u} S_{m}a\|_{L^2} &\leqslant C\, \sum_{q' \geqslant q-2} \| \Delta_{q}\nabla u  \|_{L^{\infty}} \, \| \Delta_{q'}S_{m}a  \|_{L^2}\\ 
2^{\frac{3q}{2}}\, \|  T^{'}_{\Delta_{q}\nabla u} S_{m}a\|_{L^2} &\leqslant\,C\, 2^{\frac{3q}{2}}\, \sum_{q' \geqslant q-2} 2^{\frac{q}{2}}\, 2^{-\frac{q}{2}}\, \| \Delta_{q}\nabla u  \|_{L^{\infty}} \, 2^{-2q'}\, 2^{2q'}\, \| \Delta_{q'}S_{m}a  \|_{L^2}\\
&\leqslant\,C\, \sum_{q' \geqslant q-2} 2^{2(q-q')}\, 2^{-\frac{q}{2}}\, \| \Delta_{q}\nabla u  \|_{L^{\infty}} \, 2^{2q'}\, \| \Delta_{q'}S_{m}a  \|_{L^2}\\
\end{split}
\end{equation} 
\noindent By definition of the Besov norm, there exists a sequence $c_{q'}$ belonging to $\ell^{2}(\Z)$ such that $$ 2^{2q'}\, \| \Delta_{q'}S_{m}a  \|_{L^2} \leqslant C\, c_{q'}\, \| S_{m}a \|_{B^{2}_{2,2}}.$$
\noindent As a result, by summation on $q$, we infer that 
\begin{equation}
\begin{split}
\sum_{q \in \Z} 2^{\frac{3q}{2}}\,\|  T^{'}_{\Delta_{q}\nabla u} S_{m}a\|_{L^2} &\leqslant\,C\, \Bigl(\sum_{q \in \Z} 2^{-\frac{q}{2}}\, \| \Delta_{q}\nabla u  \|_{L^{\infty}} \, d_{q}\,\Bigr) \, \| S_{m}a \|_{B^{2}_{2,2}},  
\end{split}
\end{equation}
\noindent where the sequence $d_{q}$ stems from convolution product: $\ds{d_{q} \eqdefa \sum_{q' \geqslant q-2} 2^{2(q-q')}\, c_{q'}  }$. As defined, it is clear that, by vertue of Young's inequality, $\ds{\| d_{q}\|_{\ell^{2}(\Z)} \leqslant C}$. 
\noindent Finally, Cauchy-Schwarz inequality yields
\begin{equation}
\begin{split}
\sum_{q \in \Z} 2^{\frac{3q}{2}}\, \|  T^{'}_{\Delta_{q}\nabla u} S_{m}a\|_{L^2} &\leqslant\,C\,  \| \nabla u  \|_{B^{-\frac{1}{2}}_{\infty,2}} \, \, \| S_{m}a \|_{B^{2}_{2,2}},  
\end{split}
\end{equation}
\noindent Once again, the Bernstein's embedding $\ds{ B^{1}_{2,2} \hookrightarrow B^{-\frac{1}{2}}_{\infty,2} }$, combining with an integration in time gives
\begin{equation}
\sum_{q \in \Z} 2^{\frac{3q}{2}}\, \|  T^{'}_{\Delta_{q}\nabla u} S_{m}a\|_{L^{1}_{t}(L^2)} \leqslant\,C\, \| S_{m}a \|_{L^{\infty}_{t}(B^{2}_{2,2})}\, \| \nabla u  \|_{L^{1}_{t}(B^{1}_{2,2})}  
\end{equation}
\noindent Therefore, 
\begin{equation}
\label{estimate4bis}
\sum_{q \in \Z} 2^{\frac{3q}{2}}\, \| T^{'}_{\Delta_{q}\nabla u} S_{m}a\|_{L^{1}_{t}(L^2)} \leqslant\,C\, 2^{2m} \| a \|_{L^{\infty}_{t}(L^{2})}\, \|  u  \|_{L^{1}_{t}(\dot{H}^2)}  
\end{equation}
\vskip 0.5 cm
\noindent\textbf{Conclusion} Summing estimates (\ref{estimatebis1}), (\ref{estimatebis2}), (\ref{estimatebis3}), and (\ref{estimate4bis}) completes the proof of the Lemma.
\end{proof}

\end{document}